\documentclass[a4paper,11pt]{amsart}
\pdfoutput=1 
\usepackage{amsmath,amsthm,amssymb}
\usepackage[mathscr]{eucal}
 \usepackage{cite}
\usepackage{upgreek}
\usepackage[bookmarks=false]{hyperref}
\usepackage{enumerate}
\usepackage{amsmath}
\usepackage{mathrsfs}
\usepackage{blindtext}
\usepackage{scrextend}
\usepackage{enumitem}
\usepackage{bm} 
\addtokomafont{labelinglabel}{\sffamily}
\usepackage{color}
\usepackage[at]{easylist}
\usepackage{tikz}
\usepackage{svg}
\usepackage{graphicx}

\usepackage{comment}

\numberwithin{equation}{section}

\newtheorem{thm}{Theorem}[section]
\newtheorem*{thm*}{Theorem}
\newtheorem{lem}[thm]{Lemma}
\newtheorem*{prob*}{Problem}

\newtheorem{prop}[thm]{Proposition}
\newtheorem*{prop*}{Proposition}

\newtheorem{cor}[thm]{Corollary}
\newtheorem*{cor*}{Corollary}

\theoremstyle{definition}
\newtheorem{defn}[thm]{Definition}
\newtheorem*{defn*}{Definition}

\newtheorem{rem}[thm]{Remark}

\newtheorem{question}[thm]{Question}

\newtheorem*{question*}{Question}
\newtheorem*{Pquestion*}{Popa's question}

\newtheorem*{conv*}{Convention}

\newcommand{\N}{\mathbb{N}}

\newcommand{\C}{\mathbb{C}}

\newcommand{\G}{\mathbb{G}}

\newcommand{\X}{\mathbb{X}}
\newcommand{\E}{\mathbb{E}}
\newcommand{\V}{\mathbb{V}}
\newcommand{\M}{\mathbb{M}}
\newcommand{\bS}{\mathbb{S}}

\newcommand{\cC}{\mathcal{C}}

\newcommand{\cH}{\mathcal{H}}

\newcommand{\cU}{\mathcal{U}}

\newcommand{\ee}{\varepsilon}

\newcommand{\Stab}{\operatorname{Stab}}
\newcommand{\Aut}{\operatorname{Aut}}

\newcommand{\supp}{\operatorname{supp}}

\newcommand{\Sym}{\textrm{Sym}}
\begin{document}

\title[Sofic actions on graphs]
{Sofic actions on graphs}

\author[David Gao]{David Gao}
\address{Department of Mathematical Sciences, UCSD, 9500 Gilman Dr, La Jolla, CA 92092, USA}\email{weg002@ucsd.edu}
\author[Srivatsav Kunnawalkam Elayavalli]{Srivatsav Kunnawalkam Elayavalli}
\address{Department of Mathematical Sciences, UCSD, 9500 Gilman Dr, La Jolla, CA 92092, USA}\email{srivatsav.kunnawalkam.elayavalli@vanderbilt.edu}
\urladdr{https://sites.google.com/view/srivatsavke/home}

\author[Gregory Patchell]{Gregory Patchell}
\address{Department of Mathematical Sciences, UCSD, 9500 Gilman Dr, La Jolla, CA 92092, USA}\email{gpatchel@ucsd.edu}
\urladdr{https://sites.google.com/view/gpatchel/home}

\begin{abstract}
%We introduce and study a natural new family of groups called graph wreath products, simultaneously generalizing  graph products  and generalized wreath products. 
We develop a theory of soficity for actions on graphs and obtain new applications to the study of sofic groups. We establish various examples, stability and permanence properties of sofic actions on graphs, in particular soficity is preserved by taking several natural graph join operations. We prove that an action of a group on its Cayley graph is sofic if and only if the group is sofic. We show that arbitrary actions of amenable groups on graphs are sofic. Using a graph theoretic result of E. Hrushovski, we also show that arbitrary actions of free groups on graphs are sofic. Notably we show that arbitrary actions of sofic groups on graphs, with amenable stabilizers, are sofic, settling completely an open problem from \cite{gao2024soficity}. We also show that soficity is preserved by taking limits under a natural Gromov-Hausdorff topology, generalizing prior work of the first author \cite{gao2024actionslerfgroupssets}. Our work sheds light on a family of groups called graph wreath products, simultaneously generalizing  graph products  and generalized wreath products.  Extending various prior results in this direction including soficity of generalized wreath products \cite{gao2024soficity}, B. Hayes and A. Sale \cite{HayesSale}, and soficity of graph products \cite{CHR, charlesworth2021matrix},  we show that graph wreath products are sofic if the action and acting groups are sofic. These results provide several new examples of sofic groups in a  systematic manner.
%In this article we  develop a notion of soficity for actions of countable groups on graphs. We show two  equivalent perspectives, several natural properties and examples. Notable examples include arbitrary actions of both amenable groups and free groups, and transitive actions of sofic groups with amenable stabilizers.   As applications we prove soficity for graph wreath products of sofic groups where the underlying group action is sofic. 

\end{abstract}
\maketitle

\section{Introduction}

This work builds a theory of soficity for group actions in the general setting of actions on graphs, significantly extending our previous work \cite{gao2024soficity}, and also settles certain important problems left open there. Our work has applications to the study of sofic groups, in particular in identifying new examples of sofic groups which is a very timely and interesting activity due to the utility in resolution of certain important problems and conjectures \cite{surveysofic, KLi, elekszabodirect, Bowenfinvariant, Luck, ThomDiophantine, Hayes5, jung2016ranktheoreml2invariantsfree}. The main results of this paper are listed as Theorems in the introduction.

 Our definition of soficity for group actions on graphs uses the concept of orbit approximations introduced in \cite{gao2024soficity} adapted to consider graph embeddings. In this paper graphs $\Gamma = (V,E)$ will be understood to be undirected and with no multi-edges and no self-loops. Vertices with infinite degree are permitted. In other words, $E\subset V \times V$ is any symmetric relation that excludes the diagonal. The graph complement $\Gamma^c$ has vertex set $V$ and edge set $E^c$ where $E^c = \{(v,w) : v\neq w\text{ and } (v,w)\not\in E\}$. A graph embedding $\pi:\Gamma_1=(V_1,E_2)\to\Gamma_2=(V_2,E_2)$ is an injection from $V_1$ to $V_2$ such that for all $v,w\in V_1,$ $(v,w)\in E_1$ if and only if $(\pi(v),\pi(w))\in E_2$. The group of automorphisms of $\Gamma$ will be denoted $\Aut(\Gamma)$ and will consist of all surjective graph embeddings $\pi:\Gamma\to\Gamma.$ An action of a group $G$ on a graph $\Gamma$ is a homomorphism of $G$ to $\Aut(\Gamma).$

In the following, $d$, when denoting a metric on a symmetric group of a finite set, shall always be understood to be the normalized Hamming distance, unless specified otherwise. The following is our finitary definition  of sofic actions on graphs: 

\begin{defn}\label{main defn}
Let $G$ be a countable discrete group, $\Gamma = (V,E)$ be a countable graph, $\alpha: G \curvearrowright \Gamma$ be an action, $A$ be a finite set, $\varphi: G \rightarrow \Sym(A)$ be a map (not necessarily a homomorphism):
\begin{enumerate}
    \item $\varphi$ is called \textit{unital} if $\varphi(1_G) = 1$;
    \item For a finite subset $F \subset G$ and $\epsilon > 0$, $\varphi$ is called $(F, \epsilon)$-\textit{multiplicative} if $d(\varphi(gh), \varphi(g)\varphi(h)) < \epsilon$ for all $g, h \in F$;
    \item For finite subsets $F \subset G$, $W \subset V$, and $\epsilon > 0$, $\varphi$ is called an $(F, W, \epsilon)$-\textit{orbit approximation of} $\alpha$ if there exists a finite graph $B$ and a subset $S \subset A$ such that $|S| > (1 - \epsilon)|A|$ and for each $s \in S$ there is a graph embedding $\pi_s: (W,E\mid_{W\times W}) \hookrightarrow B$ such that $\pi_{\varphi(g)s}(v) = \pi_s(\alpha(g^{-1})v)$ for all $s \in S$, $g \in F$, $v \in W$, whenever $\varphi(g)s \in S$ and $\alpha(g^{-1})v \in W$;
    \item Recall that $G$ is called \textit{sofic} if for all finite subsets $F\subset G$ and $\epsilon > 0$ there exists a finite set $A$ and a map $\varphi: G\to \Sym(A)$ which is unital, $(F,\epsilon)$-multiplicative, and $d(1, \varphi(g)) > 1 - \epsilon$ for all $g\in F\setminus\{e\}.$
    \item $\alpha$ is called \textit{sofic} if for all finite subsets $F \subset G$, $W \subset V$, and $\epsilon > 0$, there exists a finite set $A$ and a map $\varphi: G \rightarrow \Sym(A)$ which is unital, $(F, \epsilon)$-multiplicative, and an $(F, W, \epsilon)$-orbit approximation of $\alpha$.
\end{enumerate}
\end{defn}

 The infinitary definition is based on a commutative diagram in an ultraproduct construction involving the random graph on $\aleph_0$ vertices. This is discussed in Section \ref{ultra def} and proved to be equivalent to the finitary definition.

\subsection*{Permanence properties of sofic actions on graphs}
We document several basic and intuitive permamence results about sofic actions on graphs (for proofs see Proposition \ref{permanence_simple}).

\begin{prop}
    Let $\alpha:G\curvearrowright \Gamma=(V,E)$ be an action.
    \begin{enumerate}
        \item If $\pi:H\to G$ is a group homomorphism and $\alpha$ is sofic, then $\alpha\circ \pi:H\to \Aut(\Gamma)$ is a sofic action of $H$ on $\Gamma.$ Of particular interest are when $H$ is a subgroup of $G$ and when $\pi$ is a quotient map.
        \item Let $V_0 \subset V$ be invariant under the action of $G$. If $\alpha$ is sofic, then the restriction of $\alpha$ to $(V_0,E\mid_{V_0\times V_0})$ is sofic.
        \item If $G_1\subset G_2 \subset \cdots$ is an increasing sequence of subgroups of $G$ whose union is $G,$ and $\alpha:G\curvearrowright\Gamma$ restricted to each $G_i$ is sofic, then $\alpha$ is sofic.
        \item If $G\curvearrowright\Gamma$ is sofic, then $G\curvearrowright V$ is sofic in the sense of \cite{gao2024soficity}. If $G\curvearrowright V$ is sofic in the sense of \cite{gao2024soficity} then $G\curvearrowright(V,\varnothing)$ and $G\curvearrowright(V,(V \times V) \setminus D)$ are sofic. Here, $D \subset V \times V$ is the diagonal, so $(V,(V \times V) \setminus D)$ is the complete graph on the vertex set $V$.
        \item $G\curvearrowright\Gamma$ is sofic if and only if $G\curvearrowright \Gamma^c$ is sofic. We define $\Gamma^c = (V,E^c)$ where $(v,w)\in E^c$ if and only if $v\neq w$ and $(v,w)\not\in E.$
    \end{enumerate}
\end{prop}

We show a general result involving permanence of soficity for products and coproducts of sofic actions on graphs.

\begin{defn}
    Let $\Gamma_i=(V_i,E_i)$ be two graphs. Let $\Gamma_1 \prod \Gamma_2 = (V,E)$ denote any product between $\Gamma_1$ and $\Gamma_2$ of the following form:
    \begin{enumerate}
        \item[(i)] $V = V_1\times V_2$, and
        \item[(ii)] The truth value of $((v_1,v_2),(w_1,w_2)) \in E$ is logically determined by the truth values of $v_1=w_1,$ $v_2=w_2,$ $(v_1,w_1)\in E_1,$ and $(v_2,w_2)\in E_2.$
    \end{enumerate}
    The Cartesian, tensor, lexicographical, strong, co-normal, modular, and homomorphic products of graphs are all examples of graph products in our sense \cite{wiki:graphprod}. For example, if $\prod = \times,$ the Cartesian product, then $((v_1,v_2),(w_1,w_2)) \in E$ if and only if $(v_1,v_2)\in E_1$ and $w_1=w_2$ or $v_1=v_2$ and $(w_1,w_2)\in E_2$.

    We remark that if $G_i\curvearrowright \Gamma_i=(V_i,E_i)$, $i=1,2$ are actions of groups on graphs then the product action $G_1\times G_2 \curvearrowright \Gamma_1 \prod \Gamma_2$ defined by $(g_1,g_2) (v_1,v_2) = (g_1v_1,g_2v_2)$ is also a graph action, since the action of each $G_i$ preserves the truth value of equality (by acting by injections) and the truth value of being an edge (by definition of graph embedding).

    Let $\Gamma_i=(V_i,E_i)$ be two graphs. The coproduct of $\Gamma_1$ and $\Gamma_2$ is denoted $\Gamma_1\sqcup \Gamma_2$ and has vertex set $V_1\sqcup V_2$ and has edge set $E = \{(v,w) : \exists i \ v,w\in V_i \text{ and } (v,w)\in E_i\}$.
\end{defn}

\begin{prop}[see Propositions \ref{co product propo}, \ref{product propo}]
        If $G_i\curvearrowright^{\alpha_i} \Gamma_i=(V_i,E_i)$, $i=1,2$ are sofic actions of groups on graphs, then so is $G_1\times G_2 \curvearrowright^\alpha \Gamma_1 \sqcup \Gamma_2$ defined by $(g_1,g_2) v = g_iv$ if $v\in V_i$ and so is $G_1\times G_2 \curvearrowright^\alpha \Gamma_1 \prod \Gamma_2$ defined by $(g_1,g_2) (v_1,v_2) = (g_1v_1,g_2v_2)$.

\end{prop}

    The previous result applies to countable infinite (co)products by induction and the fact that for each finite set $W\subset \bigsqcup V_i,$ $W$ only intersects finitely many of the $V_i.$ We also note that in the case $G=G_1=g_2$ we can view $G\leq G\times G$ via the diagonal embedding, giving an action of $G$ on the disjoint union of $\Gamma_1$ and $\Gamma_2.$

We also obtain that sofic actions on graphs imply the acting group has a large quotient which is sofic: %The techniques all come from \cite{gao2024soficity}.

\begin{prop}[see Proposition \ref{sofic-actions-imply-sofic-group}]

    If $G\curvearrowright^\alpha\Gamma = (V,E)$ is sofic, then there is $H\trianglelefteq G$ such that $H\leq \bigcap_{v\in V}\Stab(v)$ and $G/H$ is sofic.
\end{prop}

\subsection*{Examples of sofic actions on graphs}

By using a rather standard argument involving Folner sets used prior in our work \cite{gao2024soficity}, the main difference being checking certain graph embedding criteria, we obtain that any action of an amenable group on a graph is sofic. Another family of groups that satisfy all actions are sofic, are free groups. This however, is quite tricky to obtain, in particular we needed a technical graph theoretic result of E. Hrushovski \cite{hrushovski1992extending}.

\begin{prop}[see Theorems \ref{amenable actions all}, \ref{free group actions all}]
    Any action of an amenable group or a free group on a graph is sofic.
\end{prop}

The following is perhaps the most important result of this paper, and settles one of the important open question from \cite{gao2024soficity}. 

\begin{thm}[see Theorem \ref{amenable stabilizers}]
    Let $\alpha:G\curvearrowright\Gamma=(V,E)$ be an action of a sofic group such that $G\curvearrowright V$ is transitive and $\Stab(v)$ is amenable for some/all $v\in V.$ Then $\alpha$ is sofic.
\end{thm}

We now explain the sketch of the proof of the above. The argument uses the technique of Følner tilings, i.e., partition of an amenable group into translations of finitely many tiles, each of which can be assumed to be appropriately invariant under the action of given finitely many elements of group, up to some $\epsilon$. By choosing an appropriately invariant Følner set and enlarging it to include all tiles touching it, we can show that there is a Følner sequence s.t. each set in the sequence is a union of prefixed tiles. It is known that any two sofic approximations of an amenable group are conjugate \cite{elekszabosofic}. Using this, if $H < G$ is an infinite amenable subgroup of a sofic group, then we can obtain a sequence of sofic approximations of $G$ whose underlying sets are a given Følner sequence of $H$, and on which $H$ acts by left multiplication.

Now, applying the above results, we may choose a sequence of sofic approximations of $G$ whose underlying sets are a Følner sequence of $H$, each one a union of tiles appropriately invariant under the action of the finite set $K = \{\sigma(x)^{-1}g\sigma(g^{-1}x): g \in F, x \in W\}$ for some finite subset $F \subset G$ and $W \subset V$, and where $\sigma: V = G/H \to G$ is some fixed lifting map. The finite approximating graph $B$ is then formed by collapsing each tile into a single point and connecting tiles in such a way as to approximate how vertices are connected in $\Gamma$. For some large enough subset $S$ of the Følner set on which everything ``behaves well", $\pi_s$ sends $x \in W$ to the tile corresponding to $\sigma(x)^{-1}s$. The tiles being appropriately invariant under $K$ ensures $\pi_{\varphi(g)s}(v) = \pi_s(\alpha(g^{-1})v)$ for all $s \in S$, $g \in F$, $v \in W$, whenever $\varphi(g)s \in S$ and $\alpha(g^{-1})v \in W$. And having a sequence of sofic approximations ensures for large enough index, $g \in G \setminus H$ behaves sufficiently differently from elements of $H$, which are the only ones that can send a tile to itself. This ensures $\pi_s$ are injective. Følner tilings are needed as this step requires having a prefixed finite subset of $H$ s.t. behaving sufficiently differently from them are enough to ensure injectivity, and this finite subset comes from the fixed finitely many tiles. This shows the soficity of the action.

As a corollary of the previous result we immediately obtain the following result:
\begin{cor}
    A group is sofic if and only if its left multiplication action on any/all Cayley graph is sofic.
\end{cor}

A further rich source of examples can be obtained by considering limits under a natural Gromov-Hausdorff topology:

\begin{defn}\label{g-h-top-defn}
    Let $G$ be a countable discrete group. Let $\cC$ consist of all triples $(\Gamma, v, \alpha)$, where $\Gamma = (V, E)$ is a countable graph, $v \in V$, and $\alpha$ is a transitive action of $G$ on $\Gamma$. We impose a topology, which we shall call the \emph{Gromov-Hausdorff topology}, on $\cC$, by defining $(\Gamma_n = (V_n, E_n), v_n, \alpha_n) \to (\Gamma = (V, E), v, \alpha)$ iff for every finite subset $F \subset G$, there exists $N > 0$ s.t. whenever $n \geq N$, for any $g \in F$, we have $g \in \Stab_{\alpha_n}(v_n)$ iff $g \in \Stab_\alpha(v)$; and $v_n$ is connected to $\alpha_n(g)v_n$ iff $v$ is connected to $\alpha(g)v$.
\end{defn}

We prove the following result:

\begin{thm}[see Theorem \ref{gromov-limit}]
    Let $G$ be a countable discrete group. Suppose $(\Gamma_n, v_n, \alpha_n) \to (\Gamma, v, \alpha)$ in the Gromov-Hausdorff topology and each $\alpha_n: G \curvearrowright \Gamma_n$ is a transitive sofic action. Then $\alpha$ is sofic.
\end{thm}

The above result is a generalization of Proposition 5 of \cite{gao2024actionslerfgroupssets}. Indeed, recall that for subgroups $(H_n)_{n \in \mathbb{N}}$, $H$ of a countable discrete group $G$, $H_n \to H$ in the Chabauty topology iff $1_{H_i} \to 1_H$ pointwise, as functions from $G$ to $\{0, 1\}$. As another consequence of Theorem above, we also obtain the following: 

\begin{cor}[see Corollary \ref{gromov cor}]
    Let $G$ be a countable discrete group. Let $\cH$ be the class of all $H \leq G$ s.t. every transitive action of $G$ on a graph with the stabilizer of some vertex being $H$ is sofic. Then $\cH$ is closed under taking increasing unions.
\end{cor}

\subsection*{Graph wreath products}

One can construct a natural family of groups arising from group actions on graphs, which we call graph wreath products. The reason for the terminology is because this simultaneously generalizes graph products and generalized wreath products. First recall from  See \cite{Gr90}: 
    Let $H$ be a group and $\Gamma=(V,E)$ a graph. For each $v\in V$ denote by $H_v$ an isomorphic copy of $H$. The graph product $*^\Gamma H$ is the group with presentation $$\langle H_v : v\in V \mid [h_v,h_w] = e \text{ if } h_v\in H_v, h_w\in H_w, \text{ and } (v,w)\in E \rangle.$$

    Let $G\curvearrowright\Gamma$ be an action. Then any graph product $*^\Gamma H$ is acted on by $G$ via automorphisms that permute the vertices. The graph wreath product is then the semi-direct product $*^\Gamma H \rtimes G$.

    The graph product of a tracial von Neumann algebra $M$ over a graph $\Gamma$ is defined in \cite{CaFi17}; it also admits an action of $G$ by automorphisms which permute the vertices, and we can again form the graph wreath product $*^\Gamma M \rtimes G.$ If $H$ is a sofic group (resp. if $M$ is Connes-embeddable) then $*^\Gamma H$ (resp. $*^\Gamma M$) is sofic \cite{CHR} (resp. Connes-embeddable \cite{caspers2015connes,charlesworth2021matrix}). 

We prove the following result, which extends several prior works including \cite{HayesSale, gao2024soficity, CHR, charlesworth2021matrix}:

\begin{thm}
    Fix a sofic group $H.$ Consider a sofic action of a sofic group $G\curvearrowright^\alpha\Gamma = (V,E).$ Then the graph wreath product $*^\Gamma H \rtimes G$ is sofic. Analogously, fix a Connes-embeddable tracial von Neumann algebra $M.$ Consider a sofic action of a hyperlinear group $G\curvearrowright\Gamma = (V,E).$ Then the graph wreath product $*^\Gamma M \rtimes G$ is Connes-embeddable.

\end{thm}

\textbf{Acknowledgements:} It is our great pleasure to thank Brandon Seward for their generosity and key insightful suggestion to consider Følner tilings. We also thank Henry Wilton for pointing us in the direction of \cite{hrushovski1992extending}. We thank Ben Hayes for helpful suggestions, constant encouragement and support.

\section{Main results}

%This section will cover much of the same content as \cite{gao2024soficity}, but adapted to the setting of actions on graphs. We start with definitions and then proceed to the analogous theorems.

\subsection{Ultraproduct framework}\label{ultra def} %We first define our terms and prove that the finitary and ultrapower characterizations of soficity are equivalent.
We obtain below an ultraproduct definition for sofic actions and discuss the relation with the finitary definition.

\begin{defn}
Let $(X_n)$ be a sequence of sets, $\mathcal{U}$ be a free ultrafilter on $\mathbb{N}$. Then $\prod_\mathcal{U} X_n$, the \textit{algebraic ultraproduct} of $(X_n)$, is defined as $\prod_\mathcal{U} X_n = \prod X_n/\sim$ where $f \sim g$ iff $\{n: f(n) = g(n)\} \in \mathcal{U}$. We shall write $(x_n)_\mathcal{U}$ to mean the element of $\prod_\mathcal{U} X_n$ represented by $(x_n) \in \prod X_n$. If $(A_n \subset X_n)$ is a sequence of subsets, then we shall write $(A_n)_\mathcal{U}$ to mean,
\begin{equation*}
    (A_n)_\mathcal{U} = \{(x_n)_\mathcal{U} \in \prod_\mathcal{U} X_n: \{n: x_n \in A_n\} \in \mathcal{U}\}
\end{equation*}
We may similarly define ultraproducts of graphs. If $(\Gamma_n = (V_n,E_n))$ is a sequence of graphs, then $\prod_\cU\Gamma_n$ has vertex set $\prod_\cU V_n$ and $((v_n),(w_n))$ define an edge if $\{n: (v_n,w_n)\in E_n\}\in \cU.$
\end{defn}

\begin{defn}
Let $(X_n, d_n)$ be a sequence of metric spaces, $\mathcal{U}$ be a free ultrafilter on $\mathbb{N}$. Then $\prod_\mathcal{U} (X_n, d_n)$, the \textit{metric ultraproduct} of $(X_n, d_n)$, is defined to have the underlying set $\prod X_n/\sim$ where $f \sim g$ iff $\lim_\mathcal{U} d_n(f(n), g(n)) = 0$. We shall write $(x_n)_\mathcal{U}$ to mean the element of $\prod_\mathcal{U} X_n$ represented by $(x_n) \in \prod X_n$. Then the metric on this space is defined by $d_\mathcal{U}((x_n)_\mathcal{U}, (y_n)_\mathcal{U}) = \lim_\mathcal{U} d_n(x_n, y_n)$.
\end{defn}

\begin{defn}
    Let $\Gamma_u$ be a countable graph that contains a copy of every finite graph. For example, we may take $\Gamma_u=(V_u,E_u)$ to be the disjoint union of all finite graphs, or to be the random graph on $\aleph_0$ vertices. Let $\cU$ be a free ultrafilter on $\N.$ For each $n,$ let $V_n$ be the collection of all maps $[n]\to V_u$ and $d_n$ the normalized Hamming distance on $V_n.$ Let $\V_\cU = \prod_\cU (V_n,d_n)$. We endow each $V_n$ with a second metric $\rho_n$ defined as follows: 
    $$\rho_n(f,g) = \frac{1}{n}\sum_{s\in[n]} \rho(f(s),g(s)),$$ where $$\rho(v,w) = \begin{cases}
        0 & \text{ if } v=w,\\
        1/2 &\text{ if } (v,w)\in E_u, \\
        1 & \text{ if } v\neq w \text{ and } (v,w)\not\in E_u.
    \end{cases}$$
    Then $\V_\cU$ has two metrics: $d_\cU = (d_n)_\cU$ and $\rho_\cU = (\rho_n)_\cU.$ We remark that if $d_\cU(f,g) = 1$ then $\rho_\cU(f,g) = 1/2$ if and only if 
    $$\lim_{n\to \cU}\frac{|\{s\in[n] : (f_n(s),g_n(s))\in E_u\}|}{n} = 1$$ and $\rho_\cU(f,g) = 1$ if and only if
    $$\lim_{n\to \cU}\frac{|\{s\in[n] : (f_n(s),g_n(s))\in E_u\}|}{n} = 0.$$
    
    Let $\E_\cU \subset \V_\cU\times\V_\cU$ be given by $((f_n),(g_n))\in \E_\cU$ if and only if $d_\cU((f_n),(g_n)) = 1$ and $\rho_\cU((f_n),(g_n)) = 1/2.$ We define the universal sofic graph to be $\X = (\V_\cU,\E_\cU).$ 
\end{defn}

\begin{defn}
    We define the following action $\bS_\cU$ of the universal sofic group $\prod_\cU(S_n,d)$ on the universal sofic graph by precomposition: for $g_n\in S_n$ and $f_n\in V_n,$ $\bS_\cU((g_n))((f_n)) = (f_n\circ g_n^{-1})$.
\end{defn}

\begin{rem}
    We may, as in \cite{gao2024soficity}, replace $\V$ with the Loeb measure space. In that case we would require defining liftable maps.
\end{rem}

Now we present the natural ultraproduct definition of sofic actions. 

\begin{prop}\label{ultra-prod}
Let $G$ be a countable discrete group, $\Gamma=(V,E)$ be a countable graph with the metrics $d_0$ and $\rho_0$, where $d_0$ is discrete and $$\rho_0(v,w) = \begin{cases}
        0 & \text{ if } v=w,\\
        1/2 &\text{ if } (v,w)\in E_u, \\
        1 & \text{ if } v\neq w \text{ and } (v,w)\not\in E_u,
    \end{cases}$$ and let $\alpha: G \curvearrowright \Gamma$ be an action. The following are equivalent:
\begin{enumerate}
    \item $\alpha$ is sofic;
    \item There exists a free ultrafilter $\mathcal{U}$ on $\mathbb{N}$, a group homomorphism $\varphi: G \rightarrow \prod_\mathcal{U} (S_n, d)$, and a map $\pi: \Gamma \rightarrow \mathbb{X}_\mathcal{U}$ which is $d_0$-$d_\cU$ isometric, $\rho_0$-$\rho_\cU$ isometric, and such that $\mathbb{S}_\mathcal{U}(\varphi(g))(\pi(v)) = \pi(\alpha(g)v)$ for all $v \in V$ and $g \in G$.
\end{enumerate}
\end{prop}

%The proof is similar to the proof of Proposition 2.9 in \cite{gao2024soficity}; the main new ingredient is checking graph embedding (more specifically, $\rho_0$-$\rho_\cU$ isometric). %We include the details for completeness.

\begin{proof}
    $(\Rightarrow)$ Fix increasing sequences of finite subsets $F_i \nearrow G$ and $W_i \nearrow V$. Fix a decreasing sequence $\epsilon_i \searrow 0$. For each $i$, let $\varphi_i: G \rightarrow \Sym(A_i)$ be a unital, $(F_i, \epsilon_i)$-multiplicative, and an $(F_i, W_i, \epsilon_i)$-orbit approximation of $\alpha$. By taking the Cartesian products of $A_i$ with auxiliary finite sets if necessary, we may assume $|A_i|$ is strictly increasing. By definition of $(F_i, W_i, \epsilon_i)$-orbit approximation, there exists a finite graph $B_i$ and a subset $S_i \subset A_i$ such that $|S_i| > (1 - \epsilon_i)|A_i|$ and for each $s \in S_i$ there is a graph embedding $\pi^i_s: E_i \hookrightarrow B_i$ such that $\pi^i_{\varphi_i(g)s}(v) = \pi^i_s(\alpha(g^{-1})v)$ for all $s \in S_i$, $g \in F_i$, $v \in W_i$, whenever $\varphi_i(g)s \in S_i$ and $\alpha(g^{-1})v \in W_i$. By graph-embedding $B_i$ into a universal graph $\Gamma_u=(V_u,E_u)$ we may assume each $\pi^i_s$ is a graph embedding into $\Gamma_u$. Let $\mathcal{U}$ be any free ultrafilter on $\mathbb{N}$ containing the set $\{|A_i|: i \geq 1\}$. We may then define $\varphi: G \rightarrow \prod_\mathcal{U} (S_n, d)$ by defining $\varphi(g) = (g_n)_\mathcal{U}$ with $g_n = \varphi_i(g)$ whenever $n = |A_i|$ and $g_n = 1$ otherwise. Since $\{|A_i|: i \geq 1\} \in \mathcal{U}$ and $\varphi_i$ is $(F_i, \epsilon_i)$-multiplicative, we see that $\varphi$ is a group homomorphism.

We then define $\pi: V \rightarrow \mathbb{V}_\mathcal{U}$ as follows: For each $v \in V$, $n \in \mathbb{N}$, define $\pi_{v, n}: [n] \rightarrow \mathbb{N}$ by
\begin{equation*}
    \pi_{v, n}(s) = \begin{cases}
        \pi^i_s(v), &\textrm{if }n = |A_i|\textrm{ and }v \in W_i\textrm{ and }s \in S_i\\
        0, &\textrm{otherwise,}
    \end{cases}
\end{equation*}
where we identify $A_i$ with $[|A_i|].$

For each $v \in V$, $\pi(v)$ shall be represented by the sequence of maps $\pi_{v, n}$. We observe that as $\{|A_i|: i \geq 1\} \in \mathcal{U}$, it does not matter how $\pi_{v, n}$ is defined when $n \notin \{|A_i|: i \geq 1\}$. Now, let $v \neq w \in V$, then for large enough $i$ we have $v, w \in W_i$. For $s \in S_i$, we then have $\pi_{v, |A_i|}(s) = \pi^i_s(v) \neq \pi^i_s(w) = \pi_{w, |A_i|}(s)$ as $\pi^i_s$ is a graph embedding, so injective. As $\frac{|S_i|}{|A_i|} > 1 - \epsilon_i \rightarrow 1$, $d_{|A_i|}(\pi_{v, |A_i|}, \pi_{w, |A_i|}) \geq \frac{|S_i|}{|A_i|} \rightarrow 1$, so $d_\mathcal{U}(\pi(v), \pi(w)) = 1$.

We now show that $\pi$ is also $\rho_0$-$\rho_\cU$ isometric. If $v,w\in V,$ then for all $i$ sufficiently large, $v,w\in W_i.$ Then $\pi^i_s(v) = \pi_{v,|A_i|}(s)$ and $\pi^i_s(w) = \pi_{w,|A_i|}(s)$ for all $s\in S_i.$ But $\pi_s^i$ is a graph embedding, so $(v,w)\in E$ if and only if $(\pi_{v,|A_i|}(s),\pi_{w,|A_i|}(s))\in E.$ Hence 
$$\lim_{n\to \cU}\frac{|\{s\in[n] : (\pi_{v,n}(s),\pi_{w,n}(s))\in E_u\}|}{n} = \begin{cases}
    1, & \text{ if } (v,w)\in E\\
    0, & \text{ if } (v,w)\not\in E.
\end{cases} $$
If $\rho_0(v,w) = 1/2,$ then $(v,w) \in E$ and we see from the above equality that $\rho_\cU(\pi(v),\pi(w)) = 1/2.$ Similarly $\rho_0(v,w)=1$ implies $\rho_\cU(\pi(v),\pi(w)) = 1.$

Finally, fix $v \in V$, $g \in G$. Then for large $i$, $v \in W_i$, $g^{-1} \in F_i$, and $\alpha(g)x \in W_i$. Now, for any $s \in S_i \cap \varphi_i(g^{-1})^{-1}S_i$, by definition of $\pi_{v, n}$ and $(F_i, W_i, \epsilon_i)$-orbit approximation we have,
\begin{equation*}
    \pi_{v, |A_i|}(\varphi_i(g^{-1})s) = \pi^i_{\varphi_i(g^{-1})s}(v) = \pi^i_s(\alpha(g)v) = \pi_{\alpha(g)v, |A_i|}(s)
\end{equation*}

Since $\frac{|S_i \cap \varphi_i(g^{-1})^{-1}S_i|}{|A_i|} > 1 - 2\epsilon_i \rightarrow 1$, this means the maps given by $\mathbb{S}_\mathcal{U}(\varphi(g))(\pi(v))$ and by $\pi(\alpha(g)v)$ coincide in the metric ultraproduct $\prod_\cU ([n],d_n)$, whence they are identified in $\mathbb V_\cU.$

$(\Leftarrow)$ For each $g \in G$, we shall write $\varphi(g) = (g_n)_\mathcal{U}$. Since $\varphi(1_G) = 1$, we shall in particular choose $\varphi(1_G) = (1)_\mathcal{U}$. Let $\varphi_n: G \rightarrow S_n$ be defined by $\varphi_n(g) = g_n$. Now, fix finite $F \subset G$, $W \subset X$, and $\epsilon > 0$, we shall show that there exists $n$ such that $\varphi_n$ is unital, $(F, \epsilon)$-multiplicative, and an $(F, W, \epsilon)$-orbit approximation of $\alpha$. We observe that $\varphi_n$ is unital for all $n$. Since $F$ is finite and $\varphi$ is a group homomorphism, there exists $L_1 \in \mathcal{U}$ such that for all $n \in L_1$, $\varphi_n$ is $(F, \epsilon)$-multiplicative.

Now, we observe that, in the definition of $(F, W, \epsilon)$-orbit approximation of $\alpha$, it is not necessary that $B$ is a finite graph, as, for an infinite $B$, we may simply restrict $B$ to the subgraph with vertices $\cup_{s \in S} \pi_s(W)$ and the latter graph is finite. Thus, we may set $B$ to be a universal graph $\Gamma_u$. Choose $\epsilon' > 0$ such that $1 - 4|W|^2\epsilon' - |F||W|\epsilon' \geq 1 - \epsilon$. Now, for each $v \in W$, we represent $\pi(v)$ as a sequence of maps $(\pi_{v, n})_\mathcal{U}$ with $\pi_{v, n}: [n] \rightarrow \Gamma_u = B$. We first observe that, for any $v \neq w \in W$, as $d_\mathcal{U}(\pi(v), \pi(w)) = 1$, there exists $L_{2, v, w} \in \mathcal{U}$ such that $d_n(\pi_{v, n}, \pi_{w, n}) > 1 - \epsilon'$ and $|\rho_n(\pi_{v,n},\pi_{w,n}) - \rho_0(v,w)| < \ee'$ for all $n \in L_{2, v, w}$. Let $L_2 = \cap_{v \neq w \in W} L_{2, v, w}$. Since $W$ is finite, $L_2 \in \mathcal{U}$. For each $n \in L_2$, let
\begin{equation*}
    \tilde{S}_{n,v,w} = \{s \in [n]: \pi_{v, n}(s) \neq \pi_{w, n}(s) \text{ and } (\pi_{v,n}(s),\pi_{w,n}(s))\in E_u \iff (v,w) \in E\}
\end{equation*}

In the case $(v,w)\in E,$ we see that the conditions $d_n(\pi_{v,n},\pi_{w,n}) > 1-\ee'$ and $\rho_n(\pi_{v,n},\pi_{w,n}) > 1/2 -\ee'$ imply that $|\Tilde{S}_{n,v,w}|/n > 1-4\ee'.$ In the case $v\neq w$ and $(v,w)\not\in E$, then the condition $\rho_n(\pi_{v,n},\pi_{w,n}) > 1-\ee'$ implies $|\Tilde{S}_{n,v,w}|/n > 1-2\ee'.$ Therefore, if we define $$\Tilde{S}_n = \bigcap_{v\neq w \in W} \Tilde{S}_{n,v,w},$$
then $\frac{|\tilde{S}_n|}{n} > 1 - 4|W|^2\epsilon'$. For each $s \in \tilde{S}_n$, define $\pi^n_s: W \rightarrow B$ by $\pi^n_s(v) = \pi_{v, n}(s)$. Then $\pi^n_s$ is a graph embedding for all $s \in \tilde{S}_n$.

Now, fix any $g \in F$, $v \in W$ with $\alpha(g^{-1})v \in W$. Recall that $\mathbb{S}_\mathcal{U}(\varphi(g^{-1}))(\pi(v)) = \pi(\alpha(g^{-1})v)$. $\mathbb{S}_\mathcal{U}(\varphi(g^{-1}))(\pi(v))$ is represented by the sequence of maps $\pi_{v, n} \circ \varphi_n(g)$ while $\pi(\alpha(g^{-1})v)$ is represented by the sequence of maps $\pi_{\alpha(g^{-1})v, n}$. Thus, there exists $L_{3, g, v} \in \mathcal{U}$ such that $d_n(\pi_{v, n} \circ \varphi_n(g), \pi_{\alpha(g^{-1})v, n}) < \epsilon'$ for all $n \in L_{3, g, v}$. Let $L_3 = \cap_{g \in F, v \in W \cap \alpha(g)W} L_{3, g, v}$. Again, as $F$ and $W$ are finite, $L_3 \in \mathcal{U}$. For any $n \in L_2 \cap L_3$, define,
\begin{equation*}
    S_n = \tilde{S}_n \cap \bigcap_{g \in F, v \in W \cap \alpha(g)W} \{s \in [n]: [\pi_{v, n} \circ \varphi_n(g)](s) = \pi_{\alpha(g^{-1})v, n}(s)\}
\end{equation*}

Then $\frac{|S_n|}{n} > 1 - 4|W|^2\epsilon' - |F||W|\epsilon' \geq 1 - \epsilon$. By definition, $\pi^n_{\varphi_n(g)s}(x) = \pi^n_s(\alpha(g^{-1})x)$ for all $s \in S_n$, $g \in F$, $x \in W$, whenever $\varphi_n(g)s \in S_n$ and $\alpha(g^{-1})x \in W$. This shows that for all $n \in L_1 \cap L_2 \cap L_3 \neq \varnothing$, $\varphi_n$ is unital, $(F, \epsilon)$-multiplicative, and an $(F, W, \epsilon)$-orbit approximation of $\alpha$.
\end{proof}

%bruh

\subsection{Properties of Sofic Actions on Graphs}

In this subsection we state some basic results about sofic actions on graphs, prove soficity for products and coproducts of sofic actions on graphs, prove that sofic actions on graphs imply the acting group has a large quotient which is sofic, and that the action of any amenable group on a graph is sofic. %The techniques all come from \cite{gao2024soficity}.

%The following facts follow quickly from the definition; we include proof sketches.

\begin{prop}\label{permanence_simple}
    Let $\alpha:G\curvearrowright \Gamma=(V,E)$ be an action.
    \begin{enumerate}
        \item If $\pi:H\to G$ is a group homomorphism and $\alpha$ is sofic, then $\alpha\circ \pi:H\to \Aut(\Gamma)$ is a sofic action of $H$ on $\Gamma.$ Of particular interest are when $H$ is a subgroup of $G$ and when $\pi$ is a quotient map.
        \item Let $V_0 \subset V$ be invariant under the action of $G$. If $\alpha$ is sofic, then the restriction of $\alpha$ to $(V_0,E\mid_{V_0\times V_0})$ is sofic.
        \item If $G_1\subset G_2 \subset \cdots$ is an increasing sequence of subgroups of $G$ whose union is $G,$ and $\alpha:G\curvearrowright\Gamma$ restricted to each $G_i$ is sofic, then $\alpha$ is sofic.
        \item If $G\curvearrowright\Gamma$ is sofic, then $G\curvearrowright V$ is sofic in the sense of \cite{gao2024soficity}. If $G\curvearrowright V$ is sofic in the sense of \cite{gao2024soficity} then $G\curvearrowright(V,\varnothing)$ and $G\curvearrowright(V,(V \times V) \setminus D)$ are sofic. Here, $D \subset V \times V$ is the diagonal, so $(V,(V \times V) \setminus D)$ is the complete graph on the vertex set $V$.
        \item $G\curvearrowright\Gamma$ is sofic if and only if $G\curvearrowright \Gamma^c$ is sofic. We define $\Gamma^c = (V,E^c)$ where $(v,w)\in E^c$ if and only if $v\neq w$ and $(v,w)\not\in E.$
    \end{enumerate}
\end{prop}

\begin{proof}
    \textcolor{white}{.}
    \begin{enumerate}
        \item Let $F\subset H$ and $W\subset V$ be finite subsets and $\ee > 0.$ Since $\alpha$ is sofic, there exist a set $A$ and a map $\varphi:G \to \Sym(A)$ such that $\varphi$ is unital, $(\pi(F),\ee)$-multiplicative, and is an $(\pi(F),W,\ee)$ orbit approximation. It is immediate that the map $\varphi\circ\pi$ is unital, $(F,\ee)$-multiplicative, and is an $(F,W,\ee)$ orbit approximation.
        \item We must show that for every $F\subset G$ and $W\subset V_0$ finite subsets and $\ee>0$ that there exist a set $A$ and a map $\varphi:G \to \Sym(A)$ such that $\varphi$ is unital, $(F,\ee)$-multiplicative, and is an $(F,W,\ee)$ orbit approximation. Since $\alpha$ is sofic and $W\subset V_0 \subset V$, we always have such $A$ and $\varphi.$
        \item  Let $F\subset G$ and $W\subset V$ be finite subsets and $\ee > 0.$ Since $G$ is the increasing union of the $G_n,$ $F\subset G_N$ for some integer $N.$ Since $\alpha$ restricted to $G_N$ is sofic, we can find a set $A$ and a map $\varphi:G_N\to\Sym(A)$ such that $\varphi$ is unital, $(F,\ee)$-multiplicative, and is a $(F,W,\ee)$ orbit approximation. Any extension of $\varphi$ to $G$ will still be unital, $(F,\ee)$-multiplicative, and be a $(F,W,\ee)$ orbit approximation since the definitions of $(F,\ee)$-multiplicative and $(F,W,\ee)$ orbit approximation only concern group elements in $F.$
        \item Let $\alpha$ denote the action of $G$ on the vertex set $V$ of the graph $\Gamma.$ It suffices to look at orbit approximations since the definition of $(F,\ee)$-multiplicative is the same for sets and graphs.
        
        For the first implication, we must show that for every $F\subset G$ and $W\subset V$ finite and $\ee>0,$ there is a set $A$ and a map $\varphi:G\to\Sym(A)$ such that there exists a set $B,$ a subset $S\subset A$ such that $|S| > (1-\ee)|A|$ and for all $s\in S$ there is an injection $\pi_s : W \to B$ such that $\pi_{\varphi(g)s}(v) = \pi_s(\alpha(g^{-1})v)$ whenever $s,\varphi(g)s\in S$, $g\in F,$ and $v,\alpha(g^{-1})v \in W.$ But this is exactly saying that we require the existence of an $(F,W,\ee)$ orbit approximation, minus the requirement that the maps $\pi_s$ are graph embeddings. So the first implication is clear.

        For the second implication, we note that if $G\curvearrowright V$ is sofic in the sense of \cite{gao2024soficity} (see the previous paragraph), if $E = \emptyset$ is the edge set of $\Gamma,$ and we endow $B$ with the empty edge set, then the maps $\pi_s$ are all graph embeddings. Thus $\Gamma\curvearrowright(\Gamma,\emptyset)$ is sofic. The statement about complete graphs follows similarly.
        \item As in (4), it suffices to check the existence of orbit approximations. Let $F\subset G$ and $W\subset V$ be finite and $\ee>0.$ Then since $\G\curvearrowright \Gamma$ is sofic there is a set $A$ and a map $\varphi:G\to\Sym(A)$ such that there exists a graph $B,$ a subset $S\subset A$ such that $|S| > (1-\ee)|A|$ and for all $s\in S$ there is a graph embedding $\pi_s : (W, E|_{W\times W}) \to B$ such that $\pi_{\varphi(g)s}(v) = \pi_s(\alpha(g^{-1})v)$ whenever $s,\varphi(g)s\in S$, $g\in F,$ and $v,\alpha(g^{-1})v \in W.$ We can define the maps $\pi_s^c : (W, E^c|_{W\times W}) \to B^c$ by $\pi_s^c(v) = \pi_s(v)$ -- note that $\pi_s^c$ is still a graph embedding. Then $A, \varphi,$ $S,$ $B^c,$ and the maps $\pi_s^c$ combine to give an $(F,W,\ee)$ orbit approximation of the action of $G\curvearrowright\Gamma^c$ as required. \qedhere
    \end{enumerate}
\end{proof}

\begin{prop}\label{co product propo}
    If $G_i\curvearrowright^{\alpha_i} \Gamma_i=(V_i,E_i)$, $i=1,2$ are sofic actions of groups on graphs, then so is $G_1\times G_2 \curvearrowright^\alpha \Gamma_1 \sqcup \Gamma_2$ defined by $(g_1,g_2) v = g_iv$ if $v\in V_i$.
\end{prop}

\begin{proof}
    Fix finite subsets $F\subset G_1\times G_2,$ $W\subset V_1\sqcup V_2,$ and $\ee>0.$ Write $W_1 = W\cap V_1$ and $W_2 = W\cap V_2.$ We may also find finite subsets $F_i\subset G_i,$ $i=1,2$ such that $F\subset F_1\times F_2.$ Choose $\ee'>0$ such that $(1-\ee')^2\geq 1-\ee.$

    Since each action $G_i\curvearrowright^{\alpha_i} \Gamma_i$ is sofic, there exist, for each $i=1,2,$ maps $\varphi_i:G_i\to \Sym(A_i)$ which are unital, $(F_i,\ee')$-multiplicative, and an $(F_i,W_i,\ee')$-orbit approximation of $\alpha_i$. Define $\varphi: G_1\times G_2 \to \Sym(A_1\times A_2)$ by
    $$\varphi(g_1,g_2)(a_1,a_2) = (\varphi_1(g_1)a_1,\varphi_2(g_2)a_2).$$
    It is clear $\varphi$ is unital. We now show $\varphi$ is $(F,\ee)$-multiplicative. Indeed, for $(g_1,g_2),(h_1,h_2)\in F\subset F_1\times F_2,$ we have
    \begin{align*}
        |A_1\times A_2|\cdot(1-d(\varphi((g_1h_1,g_2h_2)),\varphi((g_1,g_2))\varphi((h_1,h_2)))) &= \prod_{i=1}^2|\{a\in A_i : \varphi_i(g_ih_i)a = \varphi_i(g_i)\varphi_i(h_i)a\}| \\
        &= \prod_{i=1}^2 |A_i| \cdot (1-d(\varphi_i(g_ih_i),\varphi_i(g_i)\varphi_i(h_i))) \\
        &> |A_1\times A_2| \cdot (1-\ee')^2\\
        &\geq |A_1\times A_2| \cdot (1-\ee),
    \end{align*}
    so $d(\varphi((g_1h_1,g_2h_2)),\varphi((g_1,g_2))\varphi((h_1,h_2))) < \ee.$

    It remains to show $\varphi$ is an $(F,W\ee)$-orbit approximation of $\alpha$. Since each $\varphi_i$ is an $(F_i,W_i,\ee')$-orbit approximation of $\alpha_i,$ there exist finite sets $S_i \subset A_i$ such that $|S_i| > (1-\ee')|A_i|$ and finite graphs $B_i$ such that there exist, for each $s\in S_i,$ graph embeddings $\pi_s^i W_i \hookrightarrow B_i$. Let $S = S_1\times S_2.$ We set $B$ to be the coproduct of $B_1$ and $B_2$; i.e., $B = B_1\sqcup B_2.$ For each $s = (s_1,s_2)\in S,$ define $\pi_s : W = W_1\sqcup W_2 \hookrightarrow B$ by
    $$\pi_s(v) = \pi_{s_i}^i(v), \text{ when } v\in W_i.$$
    It is immediate to check that $\pi_s$ is always a graph embedding. Finally, fix $s=(s_1,s_2)\in S,$ $g=(g_1,g_2) \in F\subset F_1\times F_2,$ and $v\in W_i\subset W.$ Assume $\varphi(g)s\in S$ and $\alpha(g^{-1})v\in W.$ Then $\alpha_i(g_i^{-1})v = \alpha(g_i^{-1})v = \alpha(g^{-1})v\in W_i.$ Also, $\varphi(g)s = (\varphi_1(g_1)s_1,\varphi_2(g_2)s_2)\in S = S_1\times S_2$ which implies $\varphi_i(g_i)s_i\in S_i$ for $i=1,2,$ so that
    $$\pi_{\varphi(g)s}(v) = \pi^i_{\varphi_i(g)s_i}(v) = \pi^i_{s_i}(\alpha_i(g_i^{-1})v) = \pi^i_{s_i}(\alpha(g_i^{-1})v) = \pi_s(\alpha(g^{-1})v),$$
    concluding the proof.
\end{proof}

\begin{prop}\label{product propo}
    If $G_i\curvearrowright^{\alpha_i} \Gamma_i=(V_i,E_i)$, $i=1,2$ are sofic actions of groups on graphs, then so is $G_1\times G_2 \curvearrowright^\alpha \Gamma_1 \prod \Gamma_2$ defined by $(g_1,g_2) (v_1,v_2) = (g_1v_1,g_2v_2)$.
\end{prop}

\begin{proof}
    Let $F\subset G_1\times G_2$ and $W\subset V_1\times V_2$ be finite subsets. Without loss of generality we may assume $F = F_1\times F_2$ with $F_i\subset G_i$ finite and $W = W_1\times W_2$ with $W_i\subset V_i$ finite. Let $\ee>0.$ Choose $\ee'> 0$ such that $(1-\ee')^2 \geq 1\ee.$

    Since each action $G_i\curvearrowright^{\alpha_i} \Gamma_i$ is sofic, there exist, for each $i=1,2,$ maps $\varphi_i:G_i\to \Sym(A_i)$ which are unital, $(F_i,\ee')$-multiplicative, and an $(F_i,W_i,\ee')$-orbit approximation of $\alpha_i$. Define $\varphi: G_1\times G_2 \to \Sym(A_1\times A_2)$ by
    $$\varphi(g_1,g_2)(a_1,a_2) = (\varphi_1(g_1)a_1,\varphi_2(g_2)a_2).$$
    From the previous proposition we already know $\varphi$ is unital and $(F,\ee)$-multiplicative.

    It remains to show $\varphi$ is an $(F,W\ee)$-orbit approximation of $\alpha$. Since each $\varphi_i$ is an $(F_i,W_i,\ee')$-orbit approximation of $\alpha_i,$ there exist finite sets $S_i \subset A_i$ such that $|S_i| > (1-\ee')|A_i|$ and finite graphs $B_i$ such that there exist, for each $s\in S_i,$ graph embeddings $\pi_s^i W_i \hookrightarrow B_i$. Let $S = S_1\times S_2.$ We set $B$ to be the product of $B_1$ and $B_2$; i.e., $B = B_1\prod B_2.$ Denote by $\Lambda_i$ the restriction of $\Gamma_i$ to the vertex set $W_i.$ For each $s = (s_1,s_2)\in S,$ define $\pi_s : \Lambda = \Lambda_1\prod\Lambda-2 \hookrightarrow B$ by
    $$\pi_s((v_1,v_2)) = (\pi_{s_1}^1(v_1),\pi_{s_2}^2(v_2)).$$
    It is straightforward to check that, since each $\pi_s^i$ is a graph embedding, $\pi_s$ is always a graph embedding. Finally, fix $s=(s_1,s_2)\in S,$ $g=(g_1,g_2) \in F\subset F_1\times F_2,$ and $(v_1,v_2)\in W.$ Assume $\varphi(g)s\in S$ and $\alpha(g^{-1})v\in W.$ Then
    \begin{align*}
        \pi_{\varphi(g)s}(v) &= \pi_{(\varphi_1(g_1)s_1,\varphi_2(g_2)s_2)}(v_1,v_2) \\
        &= (\pi^1_{\varphi_1(g_1)s_1}(v_1),\pi^2_{\varphi_2(g_2)s_2}(v_2)) \\
        &= (\pi^1_{s_1}(\alpha_1(g_1^{-1})v_1),\pi^2_{s_2}(\alpha_2(g_2^{-1})v_2)) \\
        &= \pi_{(s_1,s_2)}(\alpha_1(g_1^{-1}v_1),\alpha_2(g_2^{-1})v_2)\\
        &=\pi_s(\alpha(g^{-1}v)),
    \end{align*}
    finishing the proof.
\end{proof}

\begin{prop}
\label{sofic-actions-imply-sofic-group}
    If $G\curvearrowright^\alpha\Gamma = (V,E)$ is sofic, then there is $H\trianglelefteq G$ such that $H\leq \bigcap_{v\in V}\Stab(v)$ and $G/H$ is sofic.
\end{prop}

\begin{proof}
    By Theorem \ref{ultra-prod}, there are maps $\varphi:G\to\prod_\cU(S_n,d)$ and $\pi:V\to\V_\cU$ such that $\bS(\varphi(g))(\pi(v)) = \pi(\alpha(g)v)$ for all $g\in G$ and $v\in V$. Define $H = \ker(\varphi)$. Then $G/H$ is sofic since it embeds into $\prod_\cU(S_n,d).$ It remains to show that $H\subset \Stab(v)$ for all $v\in V.$ Let $g\in H.$ Then $\varphi(g) = 1,$ so that $\pi(\alpha(g)v) = \bS(\varphi(g))(\pi(v)) = \pi(v)$. But $\pi$ is an isometry, so it is injective, and thus $v = \alpha(g)v$. In other words, $g\in \Stab(v)$.
\end{proof}

The key point of the following is verifying the maps $\pi_s$ are graph embeddings.

\begin{thm}\label{amenable actions all}
    Any action of an amenable group on a graph is sofic.
\end{thm}

\begin{proof}
    Let $G$ be the acting group, $\Gamma=(V,E)$ the graph, and $\alpha$ the action. Fix finite subsets $F\subset G,$ $W\subset V$, and $\ee>0.$ Assume without loss of generality $F$ is symmetric and contains the identity of $G.$ Let $F' = F\cdot F$ and $\ee' = \frac{\ee}{2|F|}$. Since $G$ is amenable, there is a finite set $A\subset G$ such that $|A \vartriangle gA| < \ee'|A|$ for all $g\in F'.$ We now construct a unital, $(F,\ee)$-multiplicative, $(F,W,\ee)$-orbit approximation $\varphi$ of $\alpha.$ 

    For $g\in G,$ choose $\varphi(g)\in\Sym(A)$ to be any element such that $\varphi(g)a=ga$ whenever $a, ga\in A.$ It follows that $\varphi$ is unital. We now claim it is $(F,\ee)$-multiplicative. Indeed, let $g,h\in F$. For $a\in A$ such that $ha, gha\in A$, by definition we have $\varphi(g)\varphi(h)a = \varphi(gh)a.$ So, since $h^{-1}, h^{-1}g^{-1} \in F\cdot F$,
    \begin{align*}
        |A| \cdot d(\varphi(g)\varphi(h),\varphi(gh)) &\leq |\{a\in A : ha\not\in A\}| + |\{a\in A:gha\not\in A\}|\\
        &\leq |A \vartriangle h^{-1} A| + |A\vartriangle h^{-1}g^{-1}A|\\
        &< 2\ee'|A|\\
        &\leq \ee|A|.
    \end{align*}
    That is, $d(\varphi(g)\varphi(h),\varphi(gh))<\ee.$

    We now show $\varphi$ is a $(F,W,\ee)$-orbit approximation of $\alpha.$ Let $S = \{s\in A : gs\in A \ \forall g\in F\}$. Since $F$ is symmetric, $S = \cap_{g\in F}(A\cap gA)$. Since $|A\cap gA| > (1-\ee')|A|,$ we have $|S| > (1-|F|\ee')|A| > (1-\ee)|A|.$ Now let $B = (V_B,E_B)$ where $V_B = \alpha(A^{-1})\cdot W$ and $E_B = E|_{V_B\times V_B}$. Define $\pi_s : W \to V_B$ by $\pi_s(v) = \alpha(s^{-1})v.$ Each $\pi_s$ is clearly a graph embedding since $\alpha(s^{-1}) \in \Aut(\Gamma)$ and $B$ is a subgraph of $\Gamma$ containing all vertices in $\alpha(s^{-1})W$ and all edges between such vertices. 

    We also have, for all $s\in S,$ $g\in F,$ and $v\in W$ such that $\varphi(g)s\in S$ and $\alpha(g^{-1})v\in W,$
    $$\pi_{\varphi(g)s}(v) = \pi_{gs}(x) = \alpha(s^{-1}g^{-1})v = \alpha(s^{-1})\alpha(g^{-1})v = \pi_s(\alpha(g^{-1})v),$$
    where we used the fact that, since $gs\in A,$ $\varphi(g)s=gs.$ This concludes the proof.
    
\end{proof}

\begin{rem}\label{finite-graph-rem}
    Together with item 1 of Proposition \ref{permanence_simple}, this implies the action of any group on a finite graph is sofic.
\end{rem}

%\section{New Results}

We now present two main results. The first is that any action of a free group on a graph is sofic. This proof requires a graph theoretic result due to Hrushovski \cite{hrushovski1992extending}. The second is a proof proving soficity of actions with amenable stabilizers. This is a substantial generalization of Theorem 2.14 in \cite{gao2024soficity} and can be considered an optimal result in this setting. 

We begin with actions of free groups. The following result is a direct consequence of \cite{hrushovski1992extending}. We thank Henry Wilton for pointing us to this reference.

\begin{lem}\label{hru-lemma}
    Let $\Gamma = (V, E)$ be a graph, $\alpha_1, \cdots, \alpha_n \in \Aut(\Gamma)$ be finitely many automorphisms of $\Gamma$, $\Gamma_0 = (V_0, E_0) \subset \Gamma$ be a finite induced subgraph. Then there exists a finite graph $\Gamma' = (V', E')$, a graph embedding $\pi: \Gamma_0 \to \Gamma'$, and automorphisms $\phi_1, \cdots, \phi_n \in \Aut(\Gamma')$ s.t., if for some $i$ and $v \in F_0$, we have $\alpha_i(v) \in F_0$, then $\pi(\alpha_i(v)) = \phi_i(\pi(v))$.
\end{lem}

\begin{thm}\label{free group actions all}
    Any action of a free group on a graph is sofic.
\end{thm}

\begin{proof}
    By item 3 of Proposition \ref{permanence_simple}, we may assume the acting group $G$ is a finitely generated free group, i.e., $G = \mathbb{F}_n$ and we write $g_1, \cdots, g_n$ for the free generators of $G$. We denote the graph being acted on as $\Gamma = (V, E)$ and the action be $\alpha: G \to \Aut(\Gamma)$. Let $\alpha_i = \alpha(g_i) \in \Aut(\Gamma)$ for $1 \leq i \leq n$. Let $F \subset G$, $W \subset V$ be finite subsets, $\epsilon > 0$. Let $F^{-1} = \{f_1, \cdots, f_m\}$.
    
    We write $f_j = g_{i_{j, 1}}^{\epsilon_{j, 1}} \cdots g_{i_{j, k_j}}^{\epsilon_{j, k_j}}$, where $\epsilon$’s are in $\{\pm 1\}$. Let $W' \subset V$ be a finite set containing $W$ and all elements of $V$ of the form $\alpha(g_{i_{j, l}}^{\epsilon_{j, l}} \cdots g_{i_{j, k_j}}^{\epsilon_{j, k_j}})v = \alpha_{i_{j, l}}^{\epsilon_{j, l}} \cdots \alpha_{i_{j, k_j}}^{\epsilon_{j, k_j}}v$ for all $v \in W$, $1 \leq j \leq m$, $1 \leq l \leq k_j$. Then by Lemma \ref{hru-lemma}, there exists a finite graph $B$, a graph embedding $\pi: (W', E|_{W' \times W'}) \hookrightarrow B$, and automorphisms $\phi_1, \cdots, \phi_n \in \Aut(B)$ s.t. $\pi(\alpha_i(v)) = \phi_i(\pi(v))$ whenever $v \in W'$, $1 \leq i \leq n$, and $\alpha_i(v) \in W'$. Clearly, this also implies $\pi(\alpha_i^{-1}(v)) = \phi_i^{-1}(\pi(v))$ whenever $v \in W'$, $1 \leq i \leq n$, and $\alpha_i^{-1}(v) \in W'$. By sending $g_i$ to $\phi_i$, we have a group homomorphism $\psi: G \to \Aut(B)$, which induces another group homomorphism $\varphi: G \to \Sym(\Aut(B))$ given by $\varphi(g)s = \psi(g)s$.

    Note that $\Aut(B)$ is a finite set. Since $\varphi$ is a group homomorphism, it is clearly unital and $(F, \epsilon)$-multiplicative. Thus, it suffices to show it is an $(F, W, \epsilon)$-orbit approximation of $\alpha$. Let $S = \Aut(B)$. For each $s \in S$, let $\pi_s: (W, E|_{W \times W}) \hookrightarrow B$ be given by $\pi_s(v) = s^{-1}(\pi(v))$, which is clearly a graph embedding. Now, for $s \in S$, $g \in F$, $v \in W$, we have $\pi_{\varphi(g)s}(v) = \pi_{\psi(g)s}(v) = s^{-1}\psi(g^{-1})\pi(v)$. Then $g^{-1} = f_j = g_{i_{j, 1}}^{\epsilon_{j, 1}} \cdots g_{i_{j, k_j}}^{\epsilon_{j, k_j}}$ for some $j$, so,
    \begin{equation*}
        \psi(g^{-1})\pi(v) = \phi_{i_{j, 1}}^{\epsilon_{j, 1}} \cdots \phi_{i_{j, k_j}}^{\epsilon_{j, k_j}}\pi(v) = \phi_{i_{j, 1}}^{\epsilon_{j, 1}} \cdots \phi_{i_{j, k_j - 1}}^{\epsilon_{j, k_j - 1}}\pi(\alpha_{i_{j, k_j}}^{\epsilon_{j, k_j}}v) = \cdots = \pi(\alpha_{i_{j, 1}}^{\epsilon_{j, 1}} \cdots \alpha_{i_{j, k_j}}^{\epsilon_{j, k_j}}v) = \pi(\alpha(g^{-1})v)
    \end{equation*}
    
    Thus, whenever $\alpha(g^{-1})v \in E$, we have $\pi_{\varphi(g)s}(v) = s^{-1}\psi(g^{-1})\pi(v) = s^{-1}\pi(\alpha(g^{-1})x) = \pi_s(\alpha(g^{-1})x)$. This concludes the proof.
\end{proof}

We shall now show that transitive actions of sofic groups with amenable stabilizers are sofic. We thank Brandon Seward for their key suggestion to use Følner tilings and their generosity in allowing us to use their idea here. We need to start with a few lemmas:

Recall the following definition and result from \cite{Downarowicz2015TilingsOA}:

\begin{defn}
    A \emph{tiling} $T$ of a group $G$ consists of:
    \begin{enumerate}
        \item A finite collection $S(T)$ of finite subsets of $G$, each containing the unity $e$ of $G$. Elements of $S(T)$ are called the \emph{shapes} of the tiling;
        \item For each $S \in S(T)$, a set $C(S) \subset G$, called the \emph{center set} for the shape $S$,
    \end{enumerate}
    such that $(Sc)_{S \in S(T), c \in C(S)}$ is a partition of $G$. A \emph{tile} in $T$ shall indicate either $(S, c)$ or $Sc$, for $S \in S(T)$, $c \in C(S)$.
\end{defn}

\begin{lem}[Theorem 4.3 of \cite{Downarowicz2015TilingsOA}]\label{folner-tiling}
    Let $G$ be an infinite amenable group, $K \subset G$ be a finite set, $\epsilon > 0$. Then there exists a tiling $T$ of $G$ s.t. every shape $S$ in $T$ is $(K, \epsilon)$-invariant, i.e., $\frac{|kS \triangle S|}{|S|} \leq \epsilon$ for all $k \in K$. We observe that this implies all tiles of $T$ are $(K, \epsilon)$-invariant.
\end{lem}

\begin{cor}\label{folner-seq-tiled}
    Let $G$ be an infinite amenable group. Fix a tiling $T$ of $G$. Then there exists a Følner sequence $A_n$ of $G$ s.t. $|A_n| \to \infty$ and each $A_n$ is a union of tiles in $T$.
\end{cor}

\begin{proof}
    Fix an increasing sequence of finite sets $F_n \subset G$ with $\cup_n F_n = G$. For each $n$, we shall construct a finite set $A_n$ s.t. $|A_n| \geq n$, $\frac{|gA_n \triangle A_n|}{|A_n|} \leq \frac{1}{n}$ for all $g \in F_n$, and $A_n$ is a union of tiles in $T$. Then such $A_n$ clearly satisfies the desired conclusion.
    
    Let $\{S_1, \cdots, S_n\}$ be the shapes of $T$. Let $M = \max\{|S_1|, \cdots, |S_n|\}$. Let,
    \begin{equation*}
        E = (\cup_{i=1}^n S_iS_i^{-1}) \cup F_n
    \end{equation*}
    which is a finite set. Let $\epsilon = \frac{1}{2nM|E| + n}$. Since $G$ is an infinite amenable group, there exists $B$ which is $(E, \epsilon)$-invariant and furthermore $|B| \geq n$. Let $A_n$ be the union of all tiles intersecting $B$. Since the tiles are all finite and they form a partition of $G$, we observe that $|A_n| \geq |B| \geq n$ and $A_n$ is finite.

    We now calculate an upper bound on $|A_n \setminus B|$. Observe that, if $s \in B$ belongs to a tile $t \in T$ and $Es \subset B$, then as $t$ is of the form $S_ic$ for some $1 \leq i \leq n$, $c \in C(S_i)$, and as $\cup_{i=1}^n S_iS_i^{-1} \subset E$, we have $t \subset Es \subset B$. So any tile that intersects $B$ but is not fully contained in $B$ must have its intersection with $B$ consisting only of elements in $\{s \in B: Es \not\subset B\}$. Since $B$ is $(E, \epsilon)$-invariant, we have,
    \begin{equation*}
        |\{s \in B: Es \not\subset B\}| \leq \epsilon|E||B|
    \end{equation*}

    Thus, there are at most $\epsilon|E||B|$ tiles that intersect $B$ but are not fully contained in $B$. Thus,
    \begin{equation*}
        |A_n \setminus B| \leq M\epsilon|E||B|
    \end{equation*}

    But then for any $g \in F_n$, as $F_n \subset E$ and $B$ is $(E, \epsilon)$-invariant,
    \begin{equation*}
    \begin{split}
        |gA_n \triangle A_n| &= |gA_n \setminus A_n| + |A_n \setminus gA_n|\\
        &\leq |gB \setminus A_n| + |gA_n \setminus gB| + |B \setminus gA_n| + |A_n \setminus B|\\
        &\leq |gB \setminus B| + |B \setminus gB| + 2|A_n \setminus B|\\
        &\leq \epsilon|B| + 2M\epsilon|E||B|\\
        &\leq (2M|E| + 1)\epsilon|A_n|\\
        &= \frac{|A_n|}{n}
    \end{split}
    \end{equation*}
\end{proof}

If $G$ is a sofic group, by a sequence $\varphi_n: G \to \Sym(A_n)$ being a \emph{sequence of sofic approximations}, we meant that $\varphi_n$ are all unital, and, for each finite subset $F \subset G$ and $\epsilon > 0$, there exists $N > 0$ s.t., whenever $n \geq N$, we have $\varphi_n$ is $(F, \epsilon)$-multiplicative and $d(\varphi_n(g), 1) > 1 - \epsilon$ for all $g \in F \setminus \{1_G\}$.

We thank Ben Hayes for first suggesting the idea behind the following lemma:

\begin{lem}\label{sofic-approx-extn}
    Let $G$ be a sofic group, $H \leq G$ be an amenable subgroup, $\psi_n: H \to \Sym(A_n)$ be a sequence of sofic approximations with $|A_n| \to \infty$. Then there exists a sequence of strictly increasing natural numbers $\{n_k\}_k$ and a sequence of sofic approximations $\varphi_k: G \to \Sym(A_{n_k})$ extending $\psi_{n_k}$.
\end{lem}

\begin{proof}
    Since $G$ is sofic, fix a sequence of sofic approximations $\varphi_k': G \to \Sym(B_k)$ of $G$. Choose $\{n_k\}_k$ s.t. $|A_{n_k}| \geq k|B_k|$. Then we may write $A_{n_k} = (B_k \times Q_k) \sqcup R_k$, where $|C_k| \geq k$ and $R_k < |B_k|$. We may then define $\varphi_k'': G \to \Sym(A_{n_k})$ by, for any $g \in G$, $\varphi_k''(g)(b, q) = (\varphi_k'(g)b, q)$ when $(b, q) \in B_k \times Q_k$ and $\varphi_k''(g)r = r$ when $r \in R_k$. Since $\frac{|R_k|}{|A_{n_k}|} < \frac{|B_k|}{k|B_k|} = \frac{1}{k} \to 0$, $\varphi_k''$ is a sequence of sofic approximations for $G$. Then there are two sequences of sofic approximations for $H$ on the sets $A_{n_k}$, namely $\psi_{n_k}$ and $\varphi_k''|_H$. By \cite{elekszabosofic} the two embeddings are conjugate, i.e., there exists $s_k \in \Sym(A_{n_k})$ s.t. $d(\psi_{n_k}(h), s_k\varphi_k''(h)s_k^{-1}) \to 0$ for any $h \in H$. We may then define,
    \begin{equation*}
        \varphi_k(g) = \begin{cases}
            s_k\varphi_k''(g)s_k^{-1} &, \text{ if }g \notin H\\
            \psi_{n_k}(g) &, \text{ if }g \in H
        \end{cases}
    \end{equation*}
    
    It is then easy to verify that $\varphi_k$ satisfies the desired conclusion.
\end{proof}

\begin{lem}\label{finite-stabilizers}
    Let $\alpha:G\curvearrowright\Gamma=(V,E)$ be an action of a sofic group such that $G\curvearrowright V$ is transitive and $\Stab(v)$ is finite for some/all $v\in V.$ Then $\alpha$ is sofic.
\end{lem}

\begin{proof}
    Fix some $v \in V$ and let $H = \Stab(v)$. Then we may take $V = G/H$ and the action $\alpha$ being the left-multiplication action of $G$ on $V = G/H$. Let $F \subset G$, $W \subset G/H$ be finite subsets. Fix $\epsilon > 0$. Let $q: G \rightarrow G/H$ be the natural quotient map. Fix $\sigma: G/H \rightarrow G$ an arbitrary section of $q$. Let,
    \begin{equation*}
        F' = F \cup H \cup [H \cdot (\sigma(W) \cup \sigma(W)^{-1}) \cdot H]
    \end{equation*}

    We observe that since $H$ is a subgroup, $1 \in H$, so,
    \begin{equation*}
    \begin{split}
        \sigma(W) \subset H \cdot (\sigma(W) \cup \sigma(W)^{-1}) \cdot H \subset F'\\
        \sigma(W)^{-1} \subset H \cdot (\sigma(W) \cup \sigma(W)^{-1}) \cdot H \subset F'\\
        \sigma(W) \cdot H \subset H \cdot (\sigma(W) \cup \sigma(W)^{-1}) \cdot H \subset F'\\
        H \cdot \sigma(W)^{-1} \subset H \cdot (\sigma(W) \cup \sigma(W)^{-1}) \cdot H \subset F'
    \end{split}
    \end{equation*}
    
    We also observe that, as $H$ is finite, so is $F'$. Let $\epsilon' = \frac{\epsilon}{|W| + 1}$. Since $F'$ is finite and as $G$ is sofic, there is a finite set $A$ and a unital, $(F, \epsilon)$-multiplicative map $\varphi: G \rightarrow \Sym(A)$ and a subset $S' \subset A$ s.t. $|S'| > (1 - \epsilon')|A|$, and,
    \begin{enumerate}
        \item $\varphi(g)s \neq \varphi(h)s$ for all $g, h \in F' \cdot F'$, $g \neq h$, $s \in S'$ (we observe here that $F' \subset F' \cdot F'$ as $1 \in H \subset F'$);
        \item $\varphi(gh)s = \varphi(g)\varphi(h)s$ for all $g, h \in F'$, $s \in S'$;
        \item $\varphi(g^{-1})s = \varphi(g)^{-1}s$ for all $g \in F'$, $s \in S'$.
    \end{enumerate}

    Now, on the set $S'$, we define a relation $s_1 \sim s_2$ if there exists $h \in H$ s.t. $s_1 = \varphi(h)s_2$. Since $\varphi$ is unital, $\sim$ is reflexive. As $H$ is a group, $H \subset F'$, and by the conditions on $S'$, we see that $\sim$ is symmetric and transitive. Hence, $\sim$ is an equivalence relation. Let $B$ be the graph with vertex set $V' = S'/\sim$, which is finite. We let $B$ have the edge set,
    \begin{equation*}
        E' = \{([s_1]_\sim, [s_2]_\sim): s_1, s_2 \in S' \, \text{s.t. } \exists x, y \in W \, \begin{cases}x, y\text{ are connected in }\Gamma\text{, and}\\
        \varphi(\sigma(x)^{-1}\sigma(y))s_2 = s_1\end{cases}\}
    \end{equation*}

    We need to show this is well-defined, i.e., $E'$ is a symmetric relation excluding the diagonal. Indeed, if $([s_1]_\sim, [s_2]_\sim) \in E'$, i.e., there exists $x, y \in W$ s.t. $x, y$ are connected in $\Gamma$ and $\varphi(\sigma(x)^{-1}\sigma(y))s_2 = s_1$. But then $y, x$ are connected in $\Gamma$, and by conditions on $S'$,
    \begin{equation*}
    \begin{split}
        s_1 = \varphi(\sigma(x)^{-1}\sigma(y))s_2 = \varphi(\sigma(x)^{-1})\varphi(\sigma(y))s_2 &\Rightarrow \varphi(\sigma(y))s_2 = \varphi(\sigma(x)^{-1})^{-1}s_1\\
        &\Rightarrow \varphi(\sigma(y)^{-1})^{-1}s_2 = \varphi(\sigma(x))s_1\\
        &\Rightarrow s_2 = \varphi(\sigma(y)^{-1})\varphi(\sigma(x))s_1
    \end{split}
    \end{equation*}

    So $([s_2]_\sim, [s_1]_\sim) \in E'$ and $E'$ is symmetric. To show it excludes the diagonal, let $v \in V'$. Suppose there exists $s_1, s_2 \in v$, $x, y \in W$ s.t. $x, y$ are connected in $\Gamma$ and $\varphi(\sigma(x)^{-1}\sigma(y))s_2 = s_1$. As $s_1 \sim s_2$, we have $s_2 = \varphi(h)s_1$ for some $h \in H$. Again, by conditions on $S'$,
    \begin{equation*}
    \begin{split}
        s_1 = \varphi(\sigma(x)^{-1}\sigma(y))s_2 = \varphi(\sigma(x)^{-1})\varphi(\sigma(y))s_2 &\Rightarrow \varphi(\sigma(x)^{-1})^{-1}s_1 = \varphi(\sigma(y))\varphi(h)s_1\\
        &\Rightarrow \varphi(\sigma(x))s_1 = \varphi(\sigma(y)h)s_1\\
        &\Rightarrow \sigma(x) = \sigma(y)h\\
        &\Rightarrow x = y
    \end{split}
    \end{equation*}

    But $x, y$ are connected in $\Gamma$, so $x \neq y$, a contradiction. Thus, $(v, v) \notin E'$.

    Having shown that $B$ is well-defined, we now define,
    \begin{equation*}
        S = S' \cap \bigcap_{x \in E} \varphi(\sigma(x)^{-1})^{-1}S'
    \end{equation*}

    We observe that as $|S'| > (1 - \epsilon')|A|$ and $\epsilon' = \frac{\epsilon}{|W| + 1}$, we have $|S| > (1 - \epsilon)|A|$. Now, for $s \in S$, we define $\pi_s: W \hookrightarrow V'$ to be $\pi_s(x) = [\varphi(\sigma(x)^{-1})s]_\sim$. We observe that this is injective. Indeed, assume $\pi_s(x) = \pi_s(y)$, i.e., $\varphi(\sigma(x)^{-1})s \sim \varphi(\sigma(y)^{-1})s$. Then there exists $h \in H$ s.t. $\varphi(\sigma(x)^{-1})s = \varphi(h)\varphi(\sigma(y)^{-1})s$. By conditions on $S'$ and noting that $s \in S \subset S'$, we have $\varphi(\sigma(x)^{-1})s = \varphi(h)\varphi(\sigma(y)^{-1})s = \varphi(h\sigma(y)^{-1})s$, so $\sigma(x)^{-1} = h\sigma(y)^{-1}$. But then $\sigma(x) = \sigma(y)h^{-1} \in \sigma(y)H = y$, i.e., $x = y$. This shows that $\pi_s$ must be injective.

    We now show that $x, y \in W$ are connected in $\Gamma$ iff $(\pi_s(x), \pi_s(y)) = ([\varphi(\sigma(x)^{-1})s]_\sim, [\varphi(\sigma(y)^{-1})s]_\sim) \in E'$. First, assume $x, y \in W$ are connected in $\Gamma$. As $\varphi(\sigma(y)^{-1})s \in S'$ and $s \in S'$, by conditions on $S'$,
    \begin{equation*}
    \begin{split}
        &\varphi(\sigma(y)^{-1})s = \varphi(\sigma(y))^{-1}s\\
        \Rightarrow\; &\varphi(\sigma(x)^{-1}\sigma(y))\varphi(\sigma(y)^{-1})s = \varphi(\sigma(x)^{-1})\varphi(\sigma(y))\varphi(\sigma(y))^{-1}s = \varphi(\sigma(x)^{-1})s
    \end{split}
    \end{equation*}

    So $([\varphi(\sigma(x)^{-1})s]_\sim, [\varphi(\sigma(y)^{-1})s]_\sim) \in E'$ by definition of $E'$.
    
    Conversely, suppose $([\varphi(\sigma(x)^{-1})s]_\sim, [\varphi(\sigma(y)^{-1})s]_\sim) \in E'$, i.e., there exists $s_1 \sim \varphi(\sigma(x)^{-1})s$, $s_2 \sim \varphi(\sigma(y)^{-1})s$, $x', y' \in W$ s.t. $x', y'$ are connected in $\Gamma$ and $\varphi(\sigma(x')^{-1}\sigma(y'))s_2 = s_1$. Since $s_1, s_2 \in S'$, by conditions on $S'$,
    \begin{equation*}
        s_1 = \varphi(\sigma(x')^{-1}\sigma(y'))s_2 = \varphi(\sigma(x')^{-1})\varphi(\sigma(y'))s_2 \Rightarrow \varphi(\sigma(x'))s_1 = \varphi(\sigma(x')^{-1})^{-1}s_1 = \varphi(\sigma(y'))s_2
    \end{equation*}

    Since $s_1 \sim \varphi(\sigma(x)^{-1})s$, $s_2 \sim \varphi(\sigma(y)^{-1})s$, there exists $h_1, h_2 \in H$ s.t. $s_1 = \varphi(h_1)\varphi(\sigma(x)^{-1})s$ and $s_2 = \varphi(h_2)\varphi(\sigma(y)^{-1})s$. Since $\varphi(\sigma(x)^{-1})s \in S'$ and $s \in S'$, by conditions on $S'$,
    \begin{equation*}
        \varphi(\sigma(x'))s_1 = \varphi(\sigma(x'))\varphi(h_1)\varphi(\sigma(x)^{-1})s = \varphi(\sigma(x')h_1)\varphi(\sigma(x)^{-1})s = \varphi(\sigma(x')h_1\sigma(x)^{-1})s
    \end{equation*}

    Similarly, $\varphi(\sigma(y'))s_2 = \varphi(\sigma(y')h_2\sigma(y)^{-1})s$. So, $\varphi(\sigma(x')h_1\sigma(x)^{-1})s = \varphi(\sigma(y')h_2\sigma(y)^{-1})s$. Again, noting that $s \in S'$, $\sigma(x')h_1\sigma(x)^{-1}, \sigma(y')h_2\sigma(y)^{-1} \in \sigma(W) \cdot H \cdot \sigma(W)^{-1} \subset F' \cdot F'$, by conditions on $S'$, we have $\sigma(x')h_1\sigma(x)^{-1} = \sigma(y')h_2\sigma(y)^{-1}$. Since $x', y'$ are connected in $\Gamma$, by multiplying by $\sigma(x')^{-1}$, we see that $H = \sigma(x')^{-1}x'$ is connected to $\sigma(x')^{-1}y' = \sigma(x')^{-1}\sigma(y')H$. Multiplying by $h_1^{-1} \in H$, we see that $H$ is connected to $h_1^{-1}\sigma(x')^{-1}\sigma(y')H$. Finally, multiplying by $\sigma(x)$, we see that $x = \sigma(x)H$ is connected to,
    \begin{equation*}
    \begin{split}
        \sigma(x)h_1^{-1}\sigma(x')^{-1}\sigma(y')H &= (\sigma(x')h_1\sigma(x)^{-1})^{-1}\sigma(y')H\\
        &= (\sigma(y')h_2\sigma(y)^{-1})^{-1}\sigma(y')H\\
        &= \sigma(y)h_2^{-1}\sigma(y')^{-1}\sigma(y')H\\
        &= \sigma(y)H\\
        &= y
    \end{split}
    \end{equation*}

    This shows $x, y$ are connected in $\Gamma$, so $\pi_s$ is indeed a graph embedding.

    Finally, for all $s \in S$, $g \in F$, $x \in W$, if $\varphi(g)s \in S$ and $\alpha(g^{-1})x = g^{-1}x \in W$, then by conditions on $S'$,
    \begin{equation*}
        \pi_{\varphi(g)s}(x) = [\varphi(\sigma(x)^{-1})\varphi(g)s]_\sim = [\varphi(\sigma(x)^{-1}g)s]_\sim
    \end{equation*}
    and,
    \begin{equation*}
        \pi_s(g^{-1}x) = [\varphi(\sigma(g^{-1}x)^{-1})s]_\sim
    \end{equation*}

    It now suffices to prove $\varphi(\sigma(x)^{-1}g)s \sim \varphi(\sigma(g^{-1}x)^{-1})s$. Let $h = \sigma(x)^{-1}g\sigma(g^{-1}x)$, which we observe is an element of $H$. Since $g^{-1}x \in W$, by conditions on $S'$,
    \begin{equation*}
        \varphi(h)\varphi(\sigma(g^{-1}x)^{-1})s = \varphi(h\sigma(g^{-1}x)^{-1})s = \varphi(\sigma(x)^{-1}g)s
    \end{equation*}

    This shows $\varphi$ is an $(F, W, \epsilon)$-orbit approximation and concludes the proof.
\end{proof}

\begin{thm}
\label{amenable stabilizers}
    Let $\alpha:G\curvearrowright\Gamma=(V,E)$ be an action of a sofic group such that $G\curvearrowright V$ is transitive and $\Stab(v)$ is amenable for some/all $v\in V.$ Then $\alpha$ is sofic.
\end{thm}

\begin{proof}
    Fix some $v \in V$ and let $H = \Stab(v)$. Then we may take $V = G/H$ and the action $\alpha$ being the left-multiplication action of $G$ on $V = G/H$. We have $H$ is amenable and by Lemma \ref{finite-stabilizers}, we may assume it is infinite. Let $F \subset G$ and $W \subset V = G/H$ be finite sets, $\epsilon > 0$. Let $\epsilon' = \frac{\epsilon}{|W|+1}$.
    
    Fix a lifting map $\sigma: G/H \to G$, and let $K = \{\sigma(x)^{-1}g\sigma(g^{-1}x): g \in F, x \in W\}$, which we observe is a finite subset of $H$. Let $\epsilon_1 > 0$ be chosen s.t. $(1 - |K|\epsilon_1) > 1 - \epsilon'$. By Lemma \ref{folner-tiling}, we may choose a tiling $T$ of $H$ such that all tiles of $T$ are $(K, \epsilon_1)$-invariant. Let the shapes of $T$ be denoted by $S_1, \cdots, S_m$ and $M = \max\{|S_1|, \cdots, |S_n|\}$. By Corollary \ref{folner-seq-tiled}, there exists a Følner sequence $A_n$ of $H$ each one a union of tiles in $T$ and $|A_n| \to \infty$. Let $\psi_n: H \to \Sym(A_n)$ be s.t. $\psi_n(h)a = ha$ for $h \in H$, $a \in A_n$, when $ha \in A_n$. Since $A_n$ is a Følner sequence, $\psi_n$ is a sequence of sofic approximations of $H$, whence by passing to a subsequence and applying Lemma \ref{sofic-approx-extn}, we may assume there is a sequence of sofic approximations $\varphi_n: G \to \Sym(A_n)$ of $G$ whose restriction to $H$ are $\psi_n$.

    Since $(1 - |K|\epsilon_1) > 1 - \epsilon'$, we may choose $\epsilon_2 > 0$ s.t. $(1 - |K|\epsilon_1)(1 - (M+1)\epsilon_2) > 1 - \epsilon'$. Since $\varphi_n$ is a sequence of sofic approximations of $G$ and $A_n$ is a Følner sequence of $H$, and by the definition of $\psi_n$ and $\varphi_n$, we have, for large enough $n$, $\varphi_n$ is $(F, \epsilon)$-multiplicative and there exists $B_n \subset A_n$ with $|B_n| \geq (1 - \epsilon_2)|A_n|$, s.t.
    \begin{enumerate}
        \item $\varphi_n(\sigma(x)^{-1}\sigma(y))b \neq \varphi_n(k)b$ for all $x, y \in W$, $x \neq y$, $k \in \cup_{i=1}^m S_iS_i^{-1}$, $b \in B_n$;
        \item $\varphi_n(\sigma(x)^{-1})b = \varphi_n(\sigma(x)^{-1}\sigma(y))\varphi_n(\sigma(y)^{-1})b \neq \varphi_n(k)\varphi_n(\sigma(y)^{-1})b$ for all $x, y \in W$, $x \neq y$, $k \in \cup_{i=1}^m S_iS_i^{-1}$, $b \in B_n$;
        \item $\varphi_n(\sigma(x)^{-1}\sigma(y))^{-1}b = \varphi_n(\sigma(y)^{-1}\sigma(x))b$ for all $x, y \in W$ and $b \in B_n$;
        \item $\varphi_n(k)b = kb$ for all $k \in \cup_{i=1}^m S_iS_i^{-1}$ and $b \in B_n$;
        \item $\varphi_n(k\sigma(x)^{-1})b = \varphi_n(k)\varphi_n(\sigma(x)^{-1})b = k\varphi_n(\sigma(x)^{-1})b$ for all $x \in W$, $k \in \cup_{i=1}^m S_iS_i^{-1}$, $b \in B_n$;
        \item $\varphi_n(\sigma(x)^{-1})\varphi_n(g)b = \varphi_n(\sigma(x)^{-1}g)b = \varphi_n(\sigma(x)^{-1}g\sigma(g^{-1}x)^{-1})\varphi_n(\sigma(g^{-1}x)^{-1})b$ for all $x \in W$, $g \in F$, $b \in B_n$;
        \item $k_1\varphi_n(\sigma(x)^{-1}\sigma(y))k_2b = \varphi_n(k_1)\varphi_n(\sigma(x)^{-1}\sigma(y))\varphi_n(k_2)b = \varphi_n(k_1\sigma(x)^{-1}\sigma(y)k_2)b$ for all $x, y \in W$, $k_1, k_2 \in \cup_{i=1}^m S_iS_i^{-1}$, $b \in B_n$;
        \item $\varphi_n(\alpha)b \neq \varphi_n(\beta)b$ for all $\alpha, \beta \in \{k_1\sigma(x)^{-1}\sigma(y)k_2: x, y \in W, k_1, k_2 \in \cup_{i=1}^m S_iS_i^{-1}\}$, $\alpha \neq \beta$, $b \in B_n$.
    \end{enumerate}

    Since $|B_n| \geq (1 - \epsilon_2)|A_n|$, there are at most $\epsilon_2|A_n|$ many tiles of $A_n$ not fully contained in $B_n$. Let $B_n' \subset B_n$ be the subset of $B_n$ with all tiles not fully contained in $B_n$ removed so that $B_n'$ is a union of tiles. Then $|B_n \setminus B_n'| \leq M\epsilon_2|A_n|$. Thus, $|B_n'| \geq (1 - (M+1)\epsilon_2)|A_n|$.

    Fix such an $n$. Let $V'$ be the collection of tiles of $T$ contained in $B_n'$, which is finite. Let $B$ be the finite graph with vertex set $V'$ and with edge set,
    \begin{equation*}
        E' = \{(v_1, v_2) \in V' \times V': \exists v_1' \in v_1 \, \exists v_2' \in v_2 \, \exists x, y \in W \, \text{s.t. }\begin{cases}x, y\text{ are connected in }\Gamma\text{, and}\\
        \varphi_n(\sigma(x)^{-1}\sigma(y))v_2' = v_1'\end{cases}\}
    \end{equation*}
    
    We need to show this is well-defined, i.e., $E'$ is a symmetric relation excluding the diagonal. Indeed, if $(v_1, v_2) \in E'$, then there exists $v_1' \in v_1$, $v_2' \in v_2$, $x, y \in W$ s.t. $x, y$ are connected in $\Gamma$ and $\varphi_n(\sigma(x)^{-1}\sigma(y))v_2' = v_1'$. But then $y, x$ are connected in $\Gamma$ and, as $v_1', v_2' \in B_n' \subset B_n$, by condition 3 in the assumptions on $B_n$,
    \begin{equation*}
        v_2' = \varphi_n(\sigma(x)^{-1}\sigma(y))^{-1}v_1' = \varphi_n(\sigma(y)^{-1}\sigma(x))v_1'
    \end{equation*}

    So $(v_2, v_1) \in E'$ and $E'$ is symmetric. To show it excludes the diagonal, let $v \in V'$. Suppose there exists $v_1', v_2' \in v$, $x, y \in W$ s.t. $x, y$ are connected in $\Gamma$ and $\varphi_n(\sigma(x)^{-1}\sigma(y))v_2' = v_1'$. Since $x \neq y$, again because $v_1', v_2' \in B_n' \subset B_n$, by conditions 1 and 4 in the assumptions on $B_n$,
    \begin{equation*}
        v_1' = \varphi_n(\sigma(x)^{-1}\sigma(y))v_2' \neq \varphi_n(k)v_2' = kv_2'
    \end{equation*}
    for any $k \in \cup_{i=1}^m S_iS_i^{-1}$. Now, since $v_1'$ and $v_2'$ are in the same tile, namely $v$, we have $v_1', v_2' \in S_ic$ for some $1 \leq i \leq m$ and $c \in C(S_i)$. But then there exists $k \in S_iS_i^{-1}$ s.t. $v_1' = kv_2'$, a contradiction. Thus, $(v, v) \notin E'$.

    Having shown that $B$ is well-defined, we now proceed to construct $S \subset A_n$ and $\pi_s$ for $s \in S$ to show that $\varphi_n$ is an $(F, W, \epsilon)$-orbit approximation of $\alpha$. Since we have already seen that $\varphi_n$ is unital and $(F, \epsilon)$-multiplicative, this would prove the theorem.

    Recall that $K = \{\sigma(x)^{-1}g\sigma(g^{-1}x): g \in F, x \in W\}$ and that all tiles in $T$ are $(K, \epsilon_1)$-invariant. Since $B_n'$ is a disjoint union of tiles, i.e., $B_n' = \sqcup_{j=1}^M T_j$ for tiles $T_j$ in $T$, for each tile $T_j$ we may define,
    \begin{equation*}
        T_j' = T_j \cap (\bigcap_{\kappa \in K} \kappa^{-1}T_j)
    \end{equation*}

    Then $|T_j'| \geq (1 - |K|\epsilon_1)|T_j|$. Let $S' = \sqcup_{j=1}^M T_j'$, so,
    \begin{equation*}
        |S'| \geq (1 - |K|\epsilon_1)|B_n'| \geq (1 - |K|\epsilon_1)(1 - (M+1)\epsilon_2)|A_n| > (1 - \epsilon')|A_n|
    \end{equation*}
    
    Finally, let,
    \begin{equation*}
        S = S' \cap (\bigcap_{x \in W} \varphi_n(\sigma(x)^{-1})^{-1}S')
    \end{equation*}
    
    Then as $|S'| > (1 - \epsilon')|A_n|$, we have $|S| > (1 - (|W|+1)\epsilon')|A_n| = (1 - \epsilon)|A_n|$.

    Now, let $\pi: B_n' \to V'$ be the map that sends $s \in B_n'$ to the corresponding tile. Then for each $s \in S$, we define $\pi_s: W \to V'$,
    \begin{equation*}
        \pi_s(x) = \pi(\varphi_n(\sigma(x)^{-1})s)
    \end{equation*}

    Note that this is well-defined as $\varphi_n(\sigma(x)^{-1})S \subset S' \subset B_n' \subset B_n$. For any $s \in S$, $g \in F$, $x \in W$, if $\varphi(g)s \in S$ and $\alpha(g^{-1})x = g^{-1}x \in W$, we have,
    \begin{equation*}
        \pi_{\varphi_n(g)s}(x) = \pi(\varphi_n(\sigma(x)^{-1})\varphi_n(g)s)
    \end{equation*}

    As $s \in B_n$, by condition 6 in the assumptions on $B_n$, we have,
    \begin{equation*}
        \varphi_n(\sigma(x)^{-1})\varphi_n(g)s = \varphi_n(\sigma(x)^{-1}g)s = \varphi_n(\sigma(x)^{-1}g\sigma(g^{-1}x)^{-1})\varphi_n(\sigma(g^{-1}x)^{-1})s
    \end{equation*}
    
    Note that $\sigma(x)^{-1}g\sigma(g^{-1}x) \in K$. Since $g^{-1}x \in W$, $s \in S$, we have $\varphi_n(\sigma(g^{-1}x)^{-1})s \in S'$. Thus, there exists $1 \leq j \leq M$ s.t. $\varphi_n(\sigma(g^{-1}x)^{-1})s \in T_j'$. By definition of $T_j'$, we have,
    \begin{equation*}
        \sigma(x)^{-1}g\sigma(g^{-1}x)^{-1}\varphi_n(\sigma(g^{-1}x)^{-1})s \in T_j \subset A_n
    \end{equation*}

    So by definition of $\psi_n$ and $\varphi_n$, we have,
    \begin{equation*}
    \begin{split}
        \varphi_n(\sigma(x)^{-1})\varphi_n(g)s &= \varphi_n(\sigma(x)^{-1}g\sigma(g^{-1}x)^{-1})\varphi_n(\sigma(g^{-1}x)^{-1})s\\
        &= \sigma(x)^{-1}g\sigma(g^{-1}x)^{-1}\varphi_n(\sigma(g^{-1}x)^{-1})s
    \end{split}
    \end{equation*}

    At the same time, as $\sigma(x)^{-1}g\sigma(g^{-1}x)^{-1}\varphi_n(\sigma(g^{-1}x)^{-1})s \in T_j$ and $\varphi_n(\sigma(g^{-1}x)^{-1})s \in T_j' \subset T_j$, they belong to the same tile, i.e., $\pi(\sigma(x)^{-1}g\sigma(g^{-1}x)^{-1}\varphi_n(\sigma(g^{-1}x)^{-1})s) = \pi(\varphi_n(\sigma(g^{-1}x)^{-1})s)$, so,
    \begin{equation*}
    \begin{split}
        \pi_{\varphi_n(g)s}(x) &= \pi(\sigma(x)^{-1}g\sigma(g^{-1}x)^{-1}\varphi_n(\sigma(g^{-1}x)^{-1})s)\\
        &= \pi(\varphi_n(\sigma(g^{-1}x)^{-1})s) = \pi_s(g^{-1}x)\\
        &= \pi_s(\alpha(g^{-1})x)
    \end{split}
    \end{equation*}

    We also need to show that $\pi_s$ is a graph embedding for any $s \in S$. Fix $s \in S$. We first show that it is injective. Assume to the contrary that $x, y \in W$, $x \neq y$, and $\pi_s(x) = \pi_s(y)$, i.e., $\pi(\varphi_n(\sigma(x)^{-1})s) = \pi(\varphi_n(\sigma(y)^{-1})s)$, i.e., $\varphi_n(\sigma(x)^{-1})s)$ and $\pi(\varphi_n(\sigma(y)^{-1})s)$ are in the same tile. Thus, there exists $k \in \cup_{i=1}^m S_iS_i^{-1}$ s.t. $\varphi_n(\sigma(x)^{-1})s = k\varphi_n(\sigma(y)^{-1})s$. Again as $s \in B_n$, by conditions 2 and 5 in the assumptions on $B_n$,
    \begin{equation*}
    \begin{split}
        \varphi_n(\sigma(x)^{-1})s &= \varphi_n(\sigma(x)^{-1}\sigma(y))\varphi_n(\sigma(y)^{-1})b\\
        &\neq \varphi_n(k)\varphi_n(\sigma(y)^{-1})b\\
        &= k\varphi_n(\sigma(y)^{-1})s\\
        &= \varphi_n(\sigma(x)^{-1})s
    \end{split}
    \end{equation*}

    This is a contradiction, which shows $\pi_s$ is injective. Lastly, we shall show $x, y \in W$ are connected in $\Gamma$ iff $(\pi_s(x), \pi_s(y)) \in E'$. Assume first that $x, y$ are connected in $\Gamma$, then by condition 2 in the assumptions on $B_n$,
    \begin{equation*}
        \varphi_n(\sigma(x)^{-1}\sigma(y))\varphi_n(\sigma(y)^{-1})s = \varphi_n(\sigma(x)^{-1})s
    \end{equation*}

    So by definition of $E'$, $(\pi_s(x), \pi_s(y)) = (\pi(\varphi_n(\sigma(x)^{-1})s), \pi(\varphi_n(\sigma(y)^{-1})s)) \in E'$.
    
    Conversely, if $(\pi_s(x), \pi_s(y)) = (\pi(\varphi_n(\sigma(x)^{-1})s), \pi(\varphi_n(\sigma(y)^{-1})s)) \in E'$, then $x \neq y$ and there exists $v_1' \in \pi(\varphi_n(\sigma(x)^{-1})s)$, $v_2' \in \pi(\varphi_n(\sigma(y)^{-1})s)$, $x', y' \in W$ s.t. $x', y'$ are connected in $\Gamma$ and $\varphi_n(\sigma(x')^{-1}\sigma(y'))v_2' = v_1'$. Since $v_1' \in \pi(\varphi_n(\sigma(x)^{-1})s)$, there exists $k_1 \in \cup_{i=1}^m S_iS_i^{-1}$ s.t. $\varphi_n(\sigma(x)^{-1})s = k_1v_1'$. Similarly, there exists $k_2 \in \cup_{i=1}^m S_iS_i^{-1}$ s.t. $v_2' = k_2\varphi_n(\sigma(y)^{-1})s$. Since $\varphi_n(\sigma(y)^{-1})s \in B_n$, by condition 7 in the assumptions on $B_n$,
    \begin{equation*}
        \varphi_n(\sigma(x)^{-1})s = k_1\varphi_n(\sigma(x')^{-1}\sigma(y'))k_2\varphi_n(\sigma(y)^{-1})s = \varphi_n(k_1\sigma(x')^{-1}\sigma(y')k_2)\varphi_n(\sigma(y)^{-1})s
    \end{equation*}

    On the other hand, by condition 2 in the assumptions on $B_n$,
    \begin{equation*}
        \varphi_n(\sigma(x)^{-1})s = \varphi_n(\sigma(x)^{-1}\sigma(y))\varphi_n(\sigma(y)^{-1})s = \varphi_n(1 \cdot \sigma(x)^{-1}\sigma(y) \cdot 1)\varphi_n(\sigma(y)^{-1})s
    \end{equation*}

    Since $1 \in \cup_{i=1}^m S_iS_i^{-1}$, by condition 8 in the assumptions on $B_n$, we must have $\sigma(x)^{-1}\sigma(y) = k_1\sigma(x')^{-1}\sigma(y')k_2$.

    Since $x'$ is connected to $y'$, by multiplying by $\sigma(x')^{-1}$, we see that $H = \sigma(x')^{-1}x'$ is connected to $\sigma(x')^{-1}y' = \sigma(x')^{-1}\sigma(y')H$. Multiplying by $k_1 \in \cup_{i=1}^m S_iS_i^{-1} \subset H$ and noting that $k_2 \in \cup_{i=1}^m S_iS_i^{-1} \subset H$, we then see that $H$ is connected to $k_1\sigma(x')^{-1}\sigma(y')H = k_1\sigma(x')^{-1}\sigma(y')k_2H = \sigma(x)^{-1}\sigma(y)H$. Finally, multiplying by $\sigma(x)$, we see that $x = \sigma(x)H$ is connected to $\sigma(x)\sigma(x)^{-1}\sigma(y)H = y$. This concludes the proof.
\end{proof}

\begin{defn}
    A Cayley graph of a group $G$ is a graph $\Gamma = (G,E)$ where, for some fixed (not necessarily finite) symmetric subset $S\subset G \setminus\{1_G\}$, $(g,h)\in E$ if and only if exists $k\in S$ such that $gk = h.$
\end{defn}

\begin{cor}
    A group is sofic if and only if its left multiplication action on any/all Cayley graph is sofic.
\end{cor}

\begin{proof}
    $(\Rightarrow)$ Suppose $G$ is sofic. The left multiplication $G\curvearrowright G$ induces a graph action $G\curvearrowright \Gamma$ on any Cayley graph $\Gamma$ of $G.$ This action is clearly transitive and has trivial stabilizers, so by Lemma \ref{finite-stabilizers} the action is sofic.

    $(\Leftarrow)$ Suppose $G\curvearrowright \Gamma$ is sofic for some Cayley graph $\Gamma$ of $G.$ Since all stabilizers are trivial, by Theorem \ref{sofic-actions-imply-sofic-group}, $G = G/\{1_G\}$ is sofic.
\end{proof}

%\section{Stability under Gromov-Hausdorff topology}
We now prove results  around limits of sofic actions   in the Gromov-Hausdorff topology.

\begin{thm}\label{gromov-limit}
    Let $G$ be a countable discrete group. Suppose $(\Gamma_n, v_n, \alpha_n) \to (\Gamma, v, \alpha)$ in the Gromov-Hausdorff topology and each $\alpha_n: G \curvearrowright \Gamma_n$ is a transitive sofic action. Then $\alpha$ is sofic.
\end{thm}

\begin{proof}
    Fix finite subsets $F \subset G$, $W \subset V$, and $\epsilon > 0$. Since $\alpha$ is transitive, we may fix a map $\sigma: W \to G$ s.t. $\alpha(\sigma(w))v = w$ for all $w \in W$. Let,
    \begin{equation*}
        F' = \sigma(W)^{-1}\sigma(W) \cup \{\sigma(w)^{-1}g\sigma(\alpha(g^{-1})w): g \in F, w \in W\}
    \end{equation*}
    
    Then by definition of the Gromov-Hausdorff topology, there exists $n > 0$ s.t. for any $g \in F'$, we have $g \in \Stab_{\alpha_n}(v_n)$ iff $g \in \Stab_\alpha(v)$; and $v_n$ is connected to $\alpha_n(g)v_n$ iff $v$ is connected to $\alpha(g)v$.

    Now, let $W_n = \{\alpha_n(\sigma(w))v_n: w \in W\}$. As $\alpha_n$ is sofic, there exists $\varphi: G \to \Sym(A)$ which is unital, $(F, \epsilon)$-multiplicative, and an $(F, W_n, \epsilon)$-orbit approximation, i.e., there exists a finite graph $B$, a subset $S \subset A$ with $|S| > (1 - \epsilon)|A|$, and for each $s \in S$, a graph embedding $\pi_s': (W_n, E_n|_{W_n \times W_n}) \hookrightarrow B$ s.t. $\pi_{\varphi(g)s}'(x) = \pi_s'(\alpha(g^{-1})x)$ for all $s \in S$, $g \in F$, $x \in W_n$, whenever $\varphi(g)s \in S$ and $\alpha(g^{-1})x \in W_n$.
    
    It suffices to show $\varphi$ is also an $(F, W, \epsilon)$-orbit approximation. For each $s \in S$, $w \in W$, we define,
    \begin{equation*}
        \pi_s(w) = \pi_s'(\alpha_n(\sigma(w))v_n)
    \end{equation*}

    To show this is a graph embedding, as $\pi_s'$ is one, it suffices to show the map $\psi: (W, E|_{W \times W}) \to (W_n, E_n|_{W_n \times W_n})$, $\psi(w) = \alpha_n(\sigma(w))v_n$, is a graph isomorphism. It is clearly surjective. For $w_1, w_2 \in W$, since $\sigma(W)^{-1}\sigma(W) \subset F'$, we have,
    \begin{equation*}
    \begin{split}
        \alpha(\sigma(w_1))v = \alpha(\sigma(w_2))v &\Leftrightarrow v = \alpha(\sigma(w_1)^{-1}\sigma(w_2))v\\
        &\Leftrightarrow v_n = \alpha_n(\sigma(w_1)^{-1}\sigma(w_2))v_n\\
        &\Leftrightarrow \alpha_n(\sigma(w_1))v_n = \alpha_n(\sigma(w_2))v_n
    \end{split}
    \end{equation*}
    i.e., $w_1 = w_2$ iff $\psi(w_1) = \psi(w_2)$, so $\psi$ is injective. By the same argument, $w_1$ and $w_2$ are connected iff $\psi(w_1)$ and $\psi(w_2)$ are connected. This shows that $\psi$ is a graph isomorphism, whence $\pi_s$ is a graph embedding.

    Now, let $s \in S$, $g \in F$, $w \in W$ s.t. $\varphi(g)s \in S$ and $\alpha(g^{-1})w \in W$. This implies $\alpha(\sigma(\alpha(g^{-1})w))v = \alpha(g^{-1})w = \alpha(g^{-1}\sigma(w))v$, so,
    \begin{equation*}
        \sigma(w)^{-1}g\sigma(\alpha(g^{-1})w) \in \Stab_\alpha(v)
    \end{equation*}

    By definition of $F'$, $\sigma(w)^{-1}g\sigma(\alpha(g^{-1})w) \in F'$, so $\sigma(w)^{-1}g\sigma(\alpha(g^{-1})w) \in \Stab_{\alpha_n}(v_n)$, i.e.,
    \begin{equation*}
        v_n = \alpha_n(\sigma(w)^{-1}g\sigma(\alpha(g^{-1})w))v_n \Rightarrow \alpha_n(g^{-1})\alpha_n(\sigma(w))v_n = \alpha_n(\sigma(\alpha(g^{-1})w))v_n
    \end{equation*}

    Thus, both $\alpha_n(\sigma(w))v_n \in W_n$ and $\alpha_n(g^{-1})\alpha_n(\sigma(w))v_n \in W_n$. We therefore have,
    \begin{equation*}
    \begin{split}
        \pi_{\varphi(g)s}(w) &= \pi_{\varphi(g)s}'(\alpha_n(\sigma(w))v_n)\\
        &= \pi_s'(\alpha_n(g^{-1})\alpha_n(\sigma(w))v_n)\\
        &= \pi_s'(\alpha_n(\sigma(\alpha(g^{-1})w))v_n)\\
        &= \pi_s(\alpha(g^{-1})w)
    \end{split}
    \end{equation*}
\end{proof}

The following is immediate:
\begin{prop}\label{chabauty-conv}
    If $(\Gamma_n, v_n, \alpha_n) \to (\Gamma, v, \alpha)$ in the Gromov-Hausdorff topology, then $\Stab_{\alpha_n}(v_n) \to \Stab_\alpha(v)$ in the Chabauty topology. If $\Gamma_n$ and $\Gamma$ are either all fully disconnected or all complete, then the converse also holds.
\end{prop}

Combining the above with item 4 of Proposition \ref{permanence_simple}, we immediately retrieve,

\begin{cor}[Propostion 5 of \cite{gao2024actionslerfgroupssets}]
    Let $(H_n)$ be a sequence of subgroups of $G.$ If $G\curvearrowright G/H_n$ is sofic for each $n$ and $H_n\to H$ in the Chabauty topology, then $G\curvearrowright G/H$ is sofic.
\end{cor}

We observe the following easy lemma, which leads to a more algebraic characterization of Gromov-Hausdorff convergence:

\begin{lem}\label{chara-of-edges}
    Let $G$ be a countable discrete group. A transitive action of $G$ on a graph with a distinguished vertex is specified by,
    \begin{enumerate}
        \item A subgroup $H \leq G$;
        \item A symmetric subset $S \subset G$, which we shall call the \emph{characteristic set of edges}, s.t. $HSH \subset S$ and $S \cap H = \varnothing$. Equivalently, a subset $S' \subset H \backslash G/H$ of the double coset space which is invariant under the involution $HgH \mapsto Hg^{-1}H$ and $H \notin S'$. ($S' = \{HgH: g \in S\}$ and $S = \bigcup S'$.)
    \end{enumerate}

    The graph being acted on is then isomorphic to $\Gamma = (G/H, E)$, where $E$ is s.t. $(gH, kH) \in E$ iff $k = gs$ for some $s \in S$, while the distinguished vertex is sent to $H$ by the isomorphism and the action of $G$ is carried over by the isomorphism to the left multiplication action of $G$ on $G/H$. Conversely, any such choice of $H$ and $S$ (or equivalently $S'$) yields a transitive action of $G$ on a graph with distinguished vertex $H$, as described above.
\end{lem}

%Most of the lemma is obvious, so we only include some parts of the proof here.

\begin{proof}
    We first note here how to retrieve $S$ from an action $G \curvearrowright \Gamma = (G/H, E)$, where the action on the vertex set is the left-multiplication action. $S$ is simply given by,
    \begin{equation*}
        S = \bigcup\{gH: (H, gH) \in E\}
    \end{equation*}

    Clearly $SH \subset S$. To show $HS \subset S$ as well, we observe that if $(H, gH) \in E$, then left multiplying by any $h \in H$ yields $(H = hH, hgH) \in E$. To show $S$ is symmetric, we observe that if $(H, gH) \in E$, then left multiplying by $g^{-1}$ yields $(g^{-1}H, H) \in E$, so as $\Gamma$ is undirected, $(H, g^{-1}H) \in E$. That $S \cap H = \varnothing$ follows from $(H, H) \notin E$.

    To show $(gH, kH) \in E$ iff $k = gs$ for some $s \in S$, we note that left multiplying by $g^{-1}$ yields $(gH, kH) \in E$ iff $(H, g^{-1}kH) \in E$ iff $g^{-1}k \in S$ iff $k = gs$ for some $s \in S$.

    For the converse, we simply note the graph and the action is well-defined. Indeed, note that as $HSH \subset S$, the relation $k = gs$ for some $s \in S$ between $gH$ and $kH$ does not depend on the choice of representatives of cosets $g$ and $k$. And, as $S$ is symmetric and $S \cap H = \varnothing$, $\Gamma = (G/H, E)$ is indeed a well-defined simple, undirected graph. Finally, the relation $k = gs$ is clearly preserved by left multiplication, i.e., for any $a \in G$, $k = gs$ iff $ak = ags$.
\end{proof}

The following proposition, which can be understood as an improvement on Proposition \ref{chabauty-conv}, is now a easy consequence of the proposition above and the definition of Gromov-Hausdorff topology:

\begin{prop}\label{gromov-characterization}
    For each $n$, let $(\Gamma_n, v_n, \alpha_n) \in \cC$ (where $\cC$ is as in Definition \ref{g-h-top-defn}). Let $(\Gamma, v, \alpha) \in \cC$. For each $n$, let $H_n = \Stab_{\alpha_n}(v_n)$ and let $S_n$ be the characteristic set of edges as given by Lemma \ref{chara-of-edges}. Similarly, let $H = \Stab_\alpha(v)$ and $S$ be the characteristic set of edges as given by Lemma \ref{chara-of-edges}. Then $(\Gamma_n, v_n, \alpha_n) \to (\Gamma, v, \alpha)$ iff $1_{H_n} \to 1_H$ pointwise (i.e., $H_n \to H$ in the Chabauty topology) and $1_{S_n} \to 1_S$ pointwise.
\end{prop}

We now have the following corollary of the proposition above and Theorem \ref{gromov-limit},

\begin{cor}\label{gromov cor}
    Let $G$ be a countable discrete group. Let $\cH$ be the class of all $H \leq G$ s.t. every transitive action of $G$ on a graph with the stabilizer of some vertex being $H$ is sofic. Then $\cH$ is closed under taking increasing unions.
\end{cor}

\begin{proof}
    Let $(H_n)_{n \in \mathbb{N}}$ be an increasing sequence in $\cH$ and $H = \cup_n H_n$. Let $(\Gamma, v, \alpha) \in \cC$ with $\Stab_\alpha(v) = H$. Let $S$ be the characteristic set of edges as given by Lemma \ref{chara-of-edges}. Then $S$ is symmetric; $HSH \subset S$, so $H_nSH_n \subset S$ as $H_n \subset H$; and $S \cap H = \varnothing$, so $S \cap H_n = \varnothing$ as $H_n \subset H$. Thus, $S$ and $H_n$ determines $(\Gamma_n, v_n, \alpha_n) \in \cC$ with $\Stab_{\alpha_n}(v_n) = H_n$ and the characteristic set of edges being $S$. As $1_{H_n} \to 1_H$ and $1_S \to 1_S$, by Proposition \ref{gromov-characterization}, $(\Gamma_n, v_n, \alpha_n) \to (\Gamma, v, \alpha)$. By assumptions, $\alpha_n$ is sofic for all $n$, so by Theorem \ref{gromov-limit}, $\alpha$ is sofic. Since $\alpha$ is an arbitrary action with the stabilizer of a vertex being $H$, this shows $H \in \cH$.
\end{proof}

Using Remark \ref{finite-graph-rem} and the corollary above, we have an immediate corollary, in analogy with the result in \cite{gao2024actionslerfgroupssets} that all actions of LERF groups are sets on sofic:

\begin{cor}
    Let $(\Gamma, v, \alpha) \in \cC$. Let $S$ be the characteristic set of edges as given by Lemma \ref{chara-of-edges}. If there exists a decreasing sequence of finite index subgroups $H_n$ and a sequence of symmetric subset $S_n \subset G$ with $H_nS_nH_n \subset S_n$ and $S_n \cap H_n = \varnothing$, s.t. $H_n \searrow \Stab_\alpha(v)$ and $1_{S_n} \to 1_S$ pointwise, then $\alpha$ is sofic. In particular, if $G$ is a group for which: whenever $(\Gamma, v, \alpha) \in \cC$ is s.t. $\Stab_\alpha(v)$ is finitely generated, then the conditions above hold, then all transitive actions of $G$ on graphs are sofic.
\end{cor}

\begin{question}
    We observe that groups that satisfy the conditions in the ``in particular" part of the corollary above must be LERF (i.e., all finitely generated subgroups are decreasing intersections of finite index subgroups). Are there nontrivial examples of such groups except finite groups?
\end{question}

\subsection{Graph Wreath Products}

In this subsection we see that, in analogy to \cite{gao2024soficity}, graph wreath products over sofic groups with a sofic action are sofic. 

\begin{defn}
    Let $H$ be a group and $\Gamma=(V,E)$ a graph. For each $v\in V$ denote by $H_v$ an isomorphic copy of $H$. The graph product $*^\Gamma H$ is the group with presentation $$\langle H_v : v\in V \mid [h_v,h_w] = e \text{ if } h_v\in H_v, h_w\in H_w, \text{ and } (v,w)\in E \rangle.$$
    See \cite{Gr90}.

    Let $G\curvearrowright\Gamma$ be an action. Then any graph product $*^\Gamma H$ is acted on by $G$ via automorphisms that permute the vertices. The graph wreath product is then the semi-direct product $*^\Gamma H \rtimes G$.

    The graph product of a tracial von Neumann algebra $M$ over a graph $\Gamma$ is defined in \cite{CaFi17}; it also admits an action of $G$ by automorphisms which permute the vertices, and we can again form the graph wreath product $*^\Gamma M \rtimes G.$
\end{defn}

\begin{rem}
    If $H$ is a sofic group (resp. if $M$ is Connes-embeddable) then $*^\Gamma H$ (resp. $*^\Gamma M$) is sofic \cite{CHR} (resp. Connes-embeddable \cite{caspers2015connes,charlesworth2021matrix}). 
\end{rem}

\begin{defn}
    Let $G$ be a countable discrete group and $A$ be a finite set. Let $\alpha_A:\Sym(A)\curvearrowright A$ be the natural permutation action. The generalized wreath product $G \wr_{\alpha_A}\Sym(A)$ is defined to be the semi-direct product $G^{\oplus A}\rtimes_{\beta_A}\Sym(A)$ where $\beta_A$ is the action $\beta_A:\Sym(A)\curvearrowright G^{\oplus A}$ induced by $\alpha_A$. We define a metric $d_{G,A}$ on this group by 
    $$d_{G,A}(g_1\sigma_1,g_2\sigma_2) = \frac{1}{|A|}|\{a\in A:\sigma_1(a)\neq\sigma_2(a) \text{ or }g_1(\sigma_1(a))\neq g_2(\sigma_2(a))\}|,$$
    where $\sigma_i\in \Sym(A)$ and $g_i\in G^{\oplus A}$ are regarded as functions from $A$ to $G$. The metric $d_{G,A}$ is invariant under both left and right multiplication.
\end{defn}

\begin{lem}[Corollary 3.5 in \cite{gao2024soficity}]\label{sofic-lemma}
    Let $G$ be a countable discrete group. Suppose there exists a free ultrafilter $\cU$ on $\N,$ a sequence of sofic groups $G_i,$ and a sequence of finite sets $A_i$ such that $G$ embeds into $\prod_\cU (G_i^{\oplus A_i}\rtimes_{\beta_{A_i}} \Sym(A_i),d_{G_i,A_i}).$ Then $G$ is sofic.
\end{lem}

A key step in the following proof is that we need to check that the maps $q_s^i$ are always group homomorphisms, which will follow from the maps $\pi_s^i$ being graph embeddings.

\begin{thm}\label{graph-wr-is-sofic}
    Fix a sofic group $H.$ Consider a sofic action of a sofic group $G\curvearrowright^\alpha\Gamma = (V,E).$ Then the graph wreath product $*^\Gamma H \rtimes G$ is sofic.
\end{thm}

\begin{proof}
    Fix increasing sequences of finite subsets $F_i\nearrow G$ and $W_i\nearrow V.$ Fix a decreasing sequence $\ee_i\searrow 0.$ For each $i,$ let $\varphi_i:H\to\Sym(A_i)$ be a unital, $(F_i,\ee_i)$-multiplicative, $(F_i,W_i,\ee_i)$-orbit approximation of $\alpha.$ Then there exist finite graphs $B_i = (V_{B_i},E_{B_i})$ and subsets $S_i\subset A_i$ such that $|S_i| > (1-\ee_i)|A_i|$ and for each $s\in S_i$ there is a graph embedding $\pi_s^i:W_i \to V_{B_i}$ such that $\pi^i_{\varphi_i(g)s}(v) = \pi_s^i(\alpha(g^{-1})v)$ for all $s\in S_i,$ $g\in F_i,$ and $v\in W_i$, whenever $\varphi_i(g)s\in S_i$ and $\alpha(g^{-1})v\in W_i.$ 

    Denote $\Gamma_i = (W_i, E|_{W_i\times W_i}).$ Let $H_i = *^{B_i} H$, which is sofic since $H$ is. Let $p_i: *^\Gamma H\to *^{\Gamma_i} H$ be the map which is identity on $*^{\Gamma_i} H$ and takes all other vertices to the identity. Let $q^i_s: *^{\Gamma_i} H \hookrightarrow *^{B_i}H = H_i$ be the homomorphism induced by the graph embedding $\pi_s^i : \Gamma_i \hookrightarrow B_i.$ Let $P^i_s = q^i_s \circ p_i.$
    
    Let the action of $G$ on $*^\Gamma H$ via $\alpha$ be denoted by $\beta$, i.e., we wish to show $*^\Gamma H\rtimes_\beta G$ is sofic. We define $\rho_i : *^\Gamma H\rtimes_\beta G \to H_i^{\oplus A_i}\rtimes_{\beta_{A_i}}\Sym(A_i)$ as follows:
    $$\rho_i(hg) = \left(\bigoplus_{s\in S_i} P^i_s(h)\oplus \bigoplus_{a\in A_i\setminus S_i}1_{H_i}\right) \cdot \varphi_i(g)$$
    where $h\in *^\Gamma H$ and $g\in G.$ Let $\cU$ be a free ultrafilter on $\N.$ Let $\rho : *^\Gamma H\rtimes G \to \prod_\cU(H_i^{\oplus A_i}\rtimes_{\beta_{A_i}}\Sym(A_i),d_{H_i,A_i})$ be given by $\rho(hg) = (\rho_i(hg))_\cU.$ 

    It now suffices to show that $\rho$ is a homomorphism and that the map $\iota: *^\Gamma H \rtimes G \to (*^\Gamma H\rtimes G)/\ker(\rho) \times G $ defined by $\iota(hg) = (hg\ker(\rho),g)$ where $h\in *^\Gamma H$ and $g\in G,$ is injective. This is because the previous lemma would imply $(*^\Gamma H\rtimes G)/\ker(\rho) \hookrightarrow \prod_\cU(H_i^{\oplus A_i}\rtimes_{\beta_{A_i}}\Sym(A_i),d_{H_i,A_i})$ is sofic.

    We start by showing $\rho$ is a homomorphism. Let $h_1,h_2\in *^\Gamma H$ and $g_1,g_2\in G.$ The set of vertices on which each $h_j,$ $j=1,2$ is supported is finite, and so for large enough $i,$ $\supp(g_2)\subset W_i,$ $\alpha(g_1)\supp(h_2)\subset W_i,$ and $g_1,g_1^{-1},g_2 \in F_i.$ Furthermore, since $\varphi_i$ is $(F_i,\ee_i)$-multiplicative, we have
    $$d(\varphi_i(g_1^{-1})^{-1},\varphi_i(g_1)) < \ee_i.$$
    Let $\Tilde{S_i} = \{s\in S_i : \varphi_i(g_1^{-1})^{-1}s = \varphi_i(g_1)s\}$. Then for large enough $i,$ $|\Tilde{S}_i| > (1-2\ee_i)|A_i|.$ Now for large enough $i,$
    \begin{equation*}
        \begin{split}
            \rho_i(h_1g_1h_2g_2) &= \rho_i(h_1\beta_{g_1}(h_2)g_1g_2)\\
    &= \left(\bigoplus_{s \in \Tilde{S}_i} P^i_s(h_1\beta_{g_1}(h_2)) \oplus \bigoplus_{a \in A_i \setminus \Tilde{S}_i} 1_{H_i}\right) \cdot \varphi_i(g_1g_2)\\
    &= \left(\bigoplus_{s \in \Tilde{S}_i} P^i_s(h_1)P^i_s(\beta_{g_1}(h_2)) \oplus \bigoplus_{a \in A_i \setminus \Tilde{S}_i} 1_{H_i}\right) \cdot \varphi_i(g_1g_2)
        \end{split}
    \end{equation*}

We now claim that $P^i_s(\beta_{g_1}(h_2)) = P^i_{\varphi_i(g_1^{-1}s}(h_2)$ for all $s\in \Tilde{S}_i\cap \varphi_i(g_1)\Tilde{S}_i$ and for all $i$ sufficiently large. Since $h\in *^\Gamma H,$ we may write $h_2 = h_{v_1}\cdots h_{v_n}$ where $h_{v_k} \in H_{v_k}$, and $H_{v_k}$ is the copy of $H$ at the vertex $v_k \in V.$ For $i$ large enough, $\alpha(g_1)\supp(h_2)$ and $\supp(h_2)$ are both contained in $W_i,$ so that $p_i(\beta(g_1)h_2) = \beta(g_1)h_2$ and $p_i(h_2) = h_2.$ 

Since $q_s^i$ is a homomorphism, it suffices to consider $h_2 \in H_v$ for some $v\in W_i \cap \alpha(g_1^{-1})W_i.$ We compute that $P_s^i(\beta(g_1)h_2) = q_s^i(\beta(g_1)h_2) = h' \in H_{\pi_s^i(\alpha(g_1)v)}$, where $h'$ is the element in $H_{\pi_s^i(\alpha(g_1)v)}$ corresponding to $h_2$ via our isomorphism of $H_v$ and $H_{\pi_s^i(\alpha(g_1)v)}$. 

On the other hand, $P^i_{\varphi_i(g_1^{-1})s}(h_2) = q^i_{\varphi_i(g_1^{-1})s}(h_2) = h'' \in H_{\pi^i_{\varphi_i(g_1^{-1})s}}(v)$ where again $h''$ corresponds to $h_2$ via the isomorphism of all copies of $H.$ But for our choices of $s$ and $i,$ we have that $\pi^i_{\varphi_i(g_1^{-1})s}(v) = \pi^i_{s}(\alpha(g_1)v).$ Therefore we actually have $h' = h''$ and hence we have shown our claim.

So,
\begin{equation*}
    \rho_i(h_1g_1h_2g_2) = \left(\bigoplus_{s \in \tilde{S}_i \cap \varphi_i(g_1)\tilde{S}_i} P^i_s(h_1)P^i_{\varphi_i(g_1^{-1})s}(h_2) \oplus \bigoplus_{a \in A_i \setminus (\tilde{S}_i \cap \varphi_i(g_1)\tilde{S}_i)} \kappa_a\right) \cdot \varphi_i(g_1g_2)
\end{equation*}
for some $\kappa_a \in H_i$. On the other hand,
\begin{equation*}
\begin{split}
    \rho_i(h_1g_1)\rho_i(h_2g_2) &= \left(\bigoplus_{s \in S_i} P^i_s(h_1) \oplus \bigoplus_{a \in A_i \setminus S_i} 1_{H_i}\right) \cdot \varphi_i(g_1) \cdot (\bigoplus_{s \in S_i} P^i_s(h_2) \oplus \bigoplus_{a \in A_i \setminus S_i} 1_{H_i}) \cdot \varphi_i(g_2)\\
    &= \left(\bigoplus_{s \in \tilde{S}_i \cap \varphi_i(g_1)\tilde{S}_i} P^i_s(h_1)P^i_{\varphi_i(g_1)^{-1}s}(h_2) \oplus \bigoplus_{a \in A_i \setminus (\tilde{S}_i \cap \varphi_i(g_1)\tilde{S}_i)} \lambda_a\right) \cdot \varphi_i(g_1)\varphi_i(g_2)
\end{split}
\end{equation*}
for some $\lambda_a \in H_i$. Thus, for large enough $i$,
\begin{equation*}
\begin{split}
    d_{H_i, A_i}(\rho_i(h_1g_1h_2g_2), \rho_i(h_1g_1)\rho_i(h_2g_2)) &\leq d(\varphi_i(g_1g_2), \varphi_i(g_1)\varphi_i(g_2)) + \frac{|A_i \setminus (\tilde{S}_i \cap \varphi_i(h_1)\tilde{S}_i)|}{|A_i|}\\
    &< \ee_i + 4\ee_i\\
    &= 5\ee_i\\
    &\rightarrow 0
\end{split}
\end{equation*}

This proves that $\rho$ is a group homomorphism. Set $N = \textrm{ker}(\rho)$. Recall we defined $\iota: *^\Gamma H\rtimes G \rightarrow (*^\Gamma H\rtimes G)/N \times H$ by $\iota(hg) = (hgN, g)$ where $h \in *^\Gamma H$ and $g \in G$. It now suffices to prove $\iota$ is injective.

Clearly, $\textrm{ker}(\iota) \subset *^\Gamma H$. Assume to the contrary that $\textrm{ker}(\iota)$ is not trivial and let $h \in \textrm{ker}(\iota) \setminus \{1\}$. Then $\rho_i(h) = (\oplus_{s \in S_i} P^i_s(h) \bigoplus \oplus_{a \in A_i \setminus S_i} 1_{H_i})$. Since $h \neq 1$, for all $i$ sufficiently large $P_s^i(h) \neq 1_{H_i}$ for all $s\in S_i.$ (This happens when $\supp(h)\subset W_i$.) Then we have
\begin{equation*}
    d_{G_i, A_i}(\rho_i(h), 1) = \frac{|S_i|}{|A_i|} = 1 - \ee_i \rightarrow 1
\end{equation*}

In particular, $\rho(h) \neq 1$, so $h \notin N$. But then $\iota(h) \neq 1$, a contradiction. This proves the claim.
    
\end{proof}

The following result can be proved in a very similar way. As such, we only provide an outline of the proof here:

\begin{thm}
    Fix a Connes-embeddable tracial von Neumann algebra $M.$ Consider a sofic action of a hyperlinear group $G\curvearrowright\Gamma = (V,E).$ Then the graph wreath product $*^\Gamma M \rtimes G$ is Connes-embeddable.
\end{thm}

\begin{proof}[Proof sketch]
    The proof follows the same outline as that of Theorem \ref{graph-wr-is-sofic}. One first observes the analogy of Lemma \ref{sofic-lemma} clearly holds in the von Neumann algebra setting by changing direct sums to tensor products and semidirect products with symmetric groups to tensoring with matrix algebras, i.e., $\prod_\cU M_i^{\otimes A_i} \otimes \M_{|A_i|}(\C)$, where $M_i$ is a sequence of Connes-embeddable tracial von Neumann algebras and $A_i$ is a sequence of finite sets, is Connes-embeddable. Now, again fix increasing sequences of finite subsets $F_i\nearrow G$ and $W_i\nearrow V.$ Fix a decreasing sequence $\ee_i\searrow 0.$ For each $i,$ let $\varphi_i:H\to\Sym(A_i)$ be a unital, $(F_i,\ee_i)$-multiplicative, $(F_i,W_i,\ee_i)$-orbit approximation of $\alpha.$ Then there exist finite graphs $B_i = (V_{B_i},E_{B_i})$ and subsets $S_i\subset A_i$ such that $|S_i| > (1-\ee_i)|A_i|$ and for each $s\in S_i$ there is a graph embedding $\pi_s^i:W_i \to V_{B_i}$ such that $\pi^i_{\varphi_i(g)s}(v) = \pi_s^i(\alpha(g^{-1})v)$ for all $s\in S_i,$ $g\in F_i,$ and $v\in W_i$, whenever $\varphi_i(g)s\in S_i$ and $\alpha(g^{-1})v\in W_i.$ 

    Denote $\Gamma_i = (W_i, E|_{W_i\times W_i}).$ Let $M_i = *^{B_i} M$, which is Connes-embeddable since $M$ is. Let $E_i: *^\Gamma M\to *^{\Gamma_i} M$ be the conditional expectation and let $q^i_s: *^{\Gamma_i} M \hookrightarrow *^{B_i}M = M_i$ be the embedding induced by the graph embedding $\pi_s^i : \Gamma_i \hookrightarrow B_i.$ Let $P^i_s = q^i_s \circ E_i.$ Then there is an embedding $\iota: *^\Gamma M \rtimes G \to (\prod_\cU M_i^{\otimes A_i} \otimes \M_{|A_i|}(\C)) \otimes L(G)$ defined by,
    \begin{equation*}
        \iota(xg) = [(\bigoplus_{s \in S_i} P^i_s(x) \oplus \bigoplus_{a \in A_i \setminus S_i} 0) \cdot \varphi_i(g)]_{i \to \cU} \otimes g
    \end{equation*}
    where we regard $\varphi_i(g) \in \Sym(A_i)$ as a permutation matrix in $\M_{|A_i|}(\C)$. That this is a well-defined embedding can be proved by the same methods as in the proof of Theorem \ref{graph-wr-is-sofic}. As both $\prod_\cU M_i^{\otimes A_i} \otimes \M_{|A_i|}(\C)$ and $L(G)$ are Connes-embeddable, this implies $*^\Gamma M \rtimes G$ is Connes-embeddable.
\end{proof}

\bibliographystyle{amsalpha}
\bibliography{inneramen}

\end{document}